\documentclass[11pt,psamsfonts]{amsart}
\usepackage{amsmath}
\usepackage{amsthm}
\usepackage{amssymb}
\usepackage{amscd}
\usepackage{amsfonts}
\usepackage{amsbsy}
\usepackage{epsfig,afterpage}
\usepackage[dvips]{psfrag}
\usepackage{color}     

\newcommand{\R}{\ensuremath{\mathbb{R}}}

\newtheorem {theorem} {Theorem} 
\newtheorem {proposition} [theorem] {Proposition}

\newtheorem {lemma} [theorem] {Lemma}
\newtheorem {definition} [theorem] {Definition}
\newtheorem {remark} {Remark}

\begin{document}

\title[On families of piecewise
smooth vector fields] {On $3-$parameter families of piecewise smooth
vector fields in the plane.}

\author[C.A. Buzzi, T. de Carvalho and M.A. Teixeira]
{Claudio A. Buzzi$^1$, Tiago de Carvalho$^2$ and\\ Marco A.
Teixeira$^3$}

\address{$^1$ IBILCE--UNESP, CEP 15054--000,
S. J. Rio Preto, S\~ao Paulo, Brazil}

\address{$^2$ FC--UNESP, CEP 17033--360,
Bauru, S\~ao Paulo, Brazil}

\address{$^3$ IMECC--UNICAMP, CEP 13081--970, Campinas,
S\~ao Paulo, Brazil}

\email{buzzi@ibilce.unesp.br}

\email{ti-car@hotmail.com}

\email{teixeira@ime.unicamp.br}


\date{}
\dedicatory{} \maketitle


\begin{abstract}

This paper is concerned with the local bifurcation analysis around
typical singularities of piecewise smooth planar dynamical systems.
Three$-$parameter families of a class of non$-$smooth vector fields
are studied and the tridimensional bifurcation diagrams are
exhibited. Our main results describe the unfolding of the so called
$fold-cusp$ singularity by means of the variation of $3$ parameters.

\end{abstract}

\section{Introduction}

NSDS's have become certainly one of the common frontiers between
Ma\-the\-matics and Physics or Engineering. Problems involving
impact or friction are piecewise$-$smooth, as are many control
systems with thresholds. Many authors have contributed to the study
of Filippov systems  (see for instance \cite{Fi} and \cite{K}). One
of the starting points for a systematic approach to the geometric
and qualitative analysis of non$-$smooth dynamical systems (NSDS's,
for short) is \cite{T1}, on smooth systems in $2-$dimensional
manifolds with boundary. The generic singularities that appear in
NSDS's, to the best of our knowledge, were first studied in
\cite{T}. Bifurcations and related pro\-blems involving or not
sliding regions were studied in papers like
\cite{Coll-Gasull-Prohens, Glendinning, diBernardo-1, diBernardo-2}.
The classification of codimension$-1$ local and some global
bifurcations for planar systems was given in \cite{Kuznetsov}. In
\cite{Marcel} codimension$-2$  singularities were discussed and it
was shown how to construct the homeomorphisms which lead to
topological equivalences between two NSDS's when the discontinuity
set is a planar smooth curve. See \cite{Marco-enciclopedia} or
\cite{diBernardo-livro} for a survey on NSDS's and references there
in.

The specific topic addressed in this paper is the qualitative
analysis of \textit{fold$-$cusp singularities} of NSDS's, where a
fold and a cusp coincide. Moreover, the bifurcation diagrams are
exhibited.

Specifically, we distinguish the following cases (see Figure
\ref{fig fold cusps}):

$\bullet$ Unfolding of an invisible fold$-$cusp singularity: 
\begin{equation}\label{eq fold-cusp 3 parametros inicio}
Z_{\lambda, \beta, \mu} = \left\{
      \begin{array}{ll}
        X_{\lambda} = \left(
              \begin{array}{c}
                   1 \\
               -x + \lambda
\end{array}
      \right)
 & \hbox{if $y \geq 0$,} \\
         Y_{\beta} = \left(
              \begin{array}{c}
               -1 \\
               - x^2 + \beta - \frac{\partial B}{\partial
                x}(x,\beta,\mu))
\end{array}
      \right)& \hbox{if $y \leq 0,$}
      \end{array}
    \right.
\end{equation}
where   $(\lambda,\beta) \in (-\lambda_0,\lambda_0) \times
(-\beta_0,\beta_0)$, with $\lambda_0
>0$ and $\beta_0>0$ sufficiently small and $B$ is a $C^2 -$\textit{bump function}
such that $B(x,\beta,\mu) =0$ if $\beta \leq 0$ and
\begin{equation}\label{eq bump function}
B(x,\beta,\mu) = \left\{
                         \begin{array}{ll}
                           0, & \hspace{-1cm}\hbox{if  $x < -\sqrt{\beta}$ or $x > 4 \sqrt{\beta}$;}  \vspace{.4cm}\\
                            B_{1}(x,\beta) + f(\beta,\mu)
                             , & \hbox{if $- \sqrt{\beta} \leq x \leq \sqrt{\beta} $;} \vspace{.4cm} \\
                            B_{2}(x,\beta) +  f(\beta,\mu)
                                        , & \hbox{if $ \sqrt{\beta} < x \leq 4 \sqrt{\beta} $.}
                         \end{array}
                       \right.
\end{equation}
if $\beta > 0$, where
$$
B_{1}(x,\beta) = \displaystyle\frac{-3}{128 \beta} \left(%
                                        \begin{array}{c}
                                           x^2 ( 208 + 3 \beta) - \\
                                           4 x \sqrt{\beta} (176 + 15 \beta) + \beta (688+93 \beta) \\
                                        \end{array}%
                                       \right),
$$

$$
B_{2}(x,\beta) = \displaystyle\frac{-1}{48 \beta} \left(%
                                        \begin{array}{c}
                                           (x- 4 \sqrt{\beta})^3 ((x^2 + \beta) ( -16 + 9 \beta) - \\
                                           2 x \sqrt{\beta} (16 + 15 \beta))
                                        \end{array}%
                                       \right)
$$
and
$$
f(\beta,\mu) = \displaystyle\frac{\mu}{48} \left(%
                                        \begin{array}{c}
                                           - 8 \beta (128+ 3 \beta) \mu
                                           \sqrt{\beta} (256 + 63 \beta) \mu - ( -64 +45 \beta ) \mu^2 - \\
                                           \beta^{-1/2} (80+3 \beta) \mu^3 + \beta^{-1}(-16 + 9  \beta) \mu^4)
                                        \end{array}%
                                       \right).
$$

$\bullet$  Unfolding of a visible fold$-$cusp singularity: 
\begin{equation}\label{eq fold-cusp 2 parametros inicio}
Z_{\lambda, \beta} = \left\{
      \begin{array}{ll}
        X_{\lambda} = \left(
              \begin{array}{c}
                   1 \\
                x - \lambda
\end{array}
      \right)
 & \hbox{if $y \geq 0$,} \\
         Y_{\beta} = \left(
              \begin{array}{c}
               1 \\
              - x^2 + \beta
\end{array}
      \right)& \hbox{if $y \leq 0,$}
      \end{array}
    \right.
\end{equation}
where   $(\lambda,\beta) \in (-\lambda_0,\lambda_0) \times
(-\beta_0,\beta_0)$, with $\lambda_0
>0$ and $\beta_0>0$ sufficiently small.


\begin{figure}[h!]
\begin{minipage}[b]{0.493\linewidth}
\begin{center}
\epsfxsize=4cm  \epsfbox{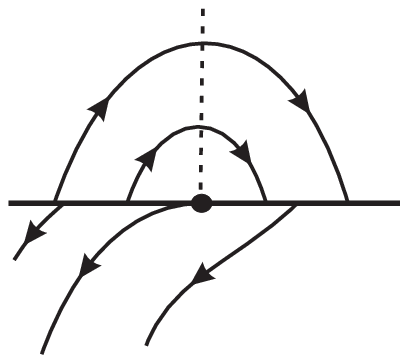}
\end{center}\end{minipage} \hfill
\begin{minipage}[b]{0.493\linewidth}
\begin{center}\psfrag{A}{$-\sqrt{\beta}$} \psfrag{B}{$0$} \psfrag{C}{$\sqrt{\beta}$}
\psfrag{D}{$$} \psfrag{E}{$3\sqrt{\beta}$}
\psfrag{F}{$4\sqrt{\beta}$}
\psfrag{G}{}
\epsfxsize=4cm \epsfbox{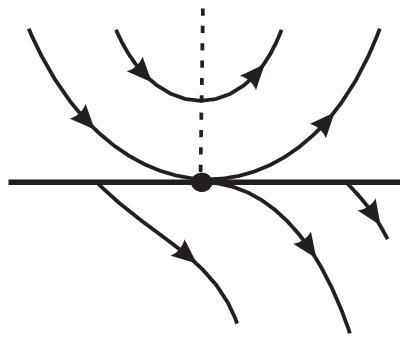}
\end{center}
\end{minipage}
\caption{\small{Fold$-$cusp singularities. Following the notation of
\cite{Marcel}, a point $p_0\in \Sigma$ is a \textbf{fold$-$cusp
singularity} of $Z=(X,Y)$ if it is a $\Sigma-$fold point of $X$ and
a $\Sigma-$cusp point of $Y$ (see precise definitions in Section
\ref{secao preliminares}).}}\label{fig fold cusps}
\end{figure}

\subsection{Setting the problem}\label{subsecao colocacao do
problema}

Denote both, $Z_{\beta,\lambda,\mu}$ in \eqref{eq fold-cusp 3
parametros inicio} and $Z_{\beta,\lambda}$ in \eqref{eq fold-cusp 2
parametros inicio}, by $Z=(X,Y)$. In short our goal is to study the
local dynamics of $Z$ consisting of two smooth vector fields $X$ and
$Y$ in $\R^2$ such that on one side of a smooth surface $\Sigma = \{
y = 0 \}$ we take $Z=X$ and on the other side $Z=Y$.

In \cite{Marcel} the analysis of the bifurcation diagram of the
$2-$parameter family
\[
W_{\mu,\epsilon} = \left\{
      \begin{array}{ll}
X_{\mu} = \left(
              \begin{array}{c}
                1 \\
                x - \mu
\end{array}
      \right)
 & \hbox{if $y \geq 0$,} \\
          Y_{\epsilon} = \left(
              \begin{array}{c}
                  -1 \\
                -x^2 + \epsilon
\end{array}
      \right)& \hbox{if $y \leq 0.$}
      \end{array}
    \right.
\] of NSDS's presenting an invisible fold$-$cusp singularity is performed.
A challenging problem is to extend the analysis of \cite{Marcel} in
answering the following question: Can we find families of NSDS's
presenting fold$-$cusp singularities whose dynamics is richer than
the family exhibited in \cite{Marcel} ? In this paper such an
extension has been carried out.  By means of the positive answer to
the previous question, we are able to say that two parameters is not
enough to explain the birth of some new topological types around
$Z_{0,0,0}$.

In fact, ours results cover the study done in \cite{Marcel} and we
can obtain the bifurcation diagram presented in \cite{Marcel}
assuming $\beta= \mu^{2}$ and $\mu \leq 0$ in \eqref{eq fold-cusp 3
parametros inicio}. For example, the configuration in Figure
\ref{fig sem simetria 1} is not observed in \cite{Marcel} and is
present at the bifurcation diagram of \eqref{eq fold-cusp 3
parametros inicio}.

\begin{figure}[h!]
\begin{minipage}[b]{0.41\linewidth}
\begin{center}
\epsfxsize=3cm  \epsfbox{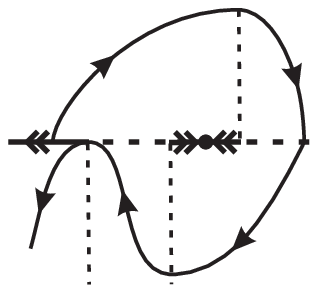}
\caption{\small{Configuration nearby $Z_{0,0,0}$ not observed in
\cite{Marcel}.}}\label{fig sem simetria 1}
\end{center}
\end{minipage} \hfill
\begin{minipage}[b]{0.58\linewidth}
\begin{center}
\epsfxsize=4cm  \epsfbox{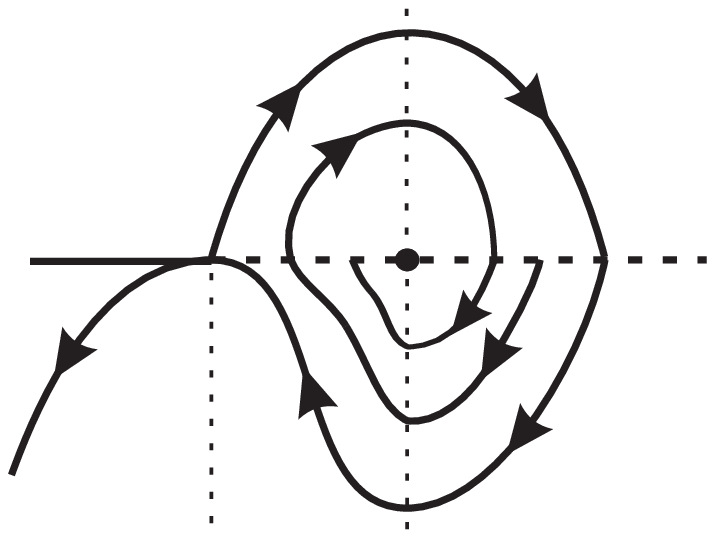} \caption{\small{The
local and the global bifurcation observed in \eqref{eq fold-cusp 3
parametros inicio} when $\beta>0$, $\lambda=\sqrt{\beta}$ and
$\mu=0$.}}\label{fig bump com o campo inicio}
\end{center}
\end{minipage}
\end{figure}

We mention two particular situations illustrated in Figure \ref{fig
bump com o campo inicio} that occur in \eqref{eq fold-cusp 3
parametros inicio} when $\beta > 0$. In this \textit{resonant}
configuration we note, simultaneously, a two$-$fold singularity
(which is a local phenomenon) and a loop passing through the visible
$\Sigma-$fold of $Y$ (which is a global phenomenon). 

\subsection{Statement of the Main Results}\label{secao enunciados}

Theorems 1, 2 and 3  pave the way for
the proof of Theorem A. Theorem $B$ is self contained. \\ 

\noindent {\bf Theorem 1.} {\it If $\mu=0$ in Equation \eqref{eq
fold-cusp 3 parametros inicio} then its bifurcation diagram in the
$(\lambda,\beta)-$plane contains essentially $17$ distinct phase
portraits (see Figure \ref{fig diag bif teo 1}).}

\vspace{.5cm}

It is easy to see that the cases covered by Theorem 1 do not
represent the full unfolding of the invisible fold$-$cusp singularity. 
Because of this, 
the next two theorems are necessary.\\

\noindent {\bf Theorem 2.} {\it If  $0 <\mu < \mu_0$ in Equation
\eqref{eq fold-cusp 3 parametros inicio} then its bifurcation
diagram in the $(\lambda,\beta)-$plane contains essentially $19$
distinct
 phase portraits (see Figure \ref{fig diag bif teo 2}).}

 \vspace{.5cm}

\noindent {\bf Theorem 3.} {\it If $-\mu_0 <\mu < 0$ in Equation
\eqref{eq fold-cusp 3 parametros inicio} then its bifurcation
diagram in the $(\lambda,\beta)-$plane contains essentially $19$
distinct
 phase portraits (see Figure \ref{fig diag bif teo 2}).}

\vspace{.5cm}

%
%
%
%
%
%
%
%
%
%

Finally, we are in position to state the main results of the paper.

\vspace{.5cm}

\noindent {\bf Theorem A.} {\it  The bifurcation diagram of Equation
\eqref{eq fold-cusp 3 parametros inicio} exhibits $55$ distinct
cases representing $23$ distinct phase portraits (see Figure
\ref{fig diagrama bifurcacao teo A}).}

\vspace{.5cm} \noindent {\bf Theorem B.} {\it The bifurcation
diagram of Equation \eqref{eq fold-cusp 2 parametros inicio}
 exhibits $11$ distinct  phase portraits (see
Figure \ref{fig diag bif teo B}).}


\vspace{.5cm}

The paper is organized as follows. In Section \ref{secao
preliminares} we present some basic elements on the theory of
NSDS's. In Sections \ref{secao prova teorema 1}, \ref{secao prova
teorema 2} and \ref{secao prova teorema 3} we pave the way for the
proofs of the main results of the paper (Theorems A and B). Section
\ref{secao prova teorema A} is devoted to prove Theorem A and
exhibit the Bifurcation Diagram of \eqref{eq fold-cusp 3 parametros
inicio}. In Section \ref{secao prova teorema B}, the proof of
Theorem B and the Bifurcation Diagram of \eqref{eq fold-cusp 2
parametros inicio} are presented and in Section \ref{secao
conclusao} some concluding remarks are discussed. In our paper we
follow basically the terminology and the approach of
\cite{Kuznetsov} or \cite{Marcel} and no one sophisticated tool is
needed.


\section{Preliminaries}\label{secao preliminares}

Let $K\subseteq \R ^{2}$ be a compact set such that $\partial K$ is
a smooth $1-$manifold and $\Sigma \subseteq K$ given by $\Sigma
=f^{-1}(0),$ where $f:K\rightarrow \R$ is a smooth function having
$0\in \R$ as a regular value (i.e. $\nabla f(p)\neq 0$, for any
$p\in f^{-1}({0}))$ such that $\partial K \cap \Sigma = \emptyset$
or $\partial K \pitchfork \Sigma$. Clearly the \textit{switching
manifold} $\Sigma$ is the separating boundary of the regions
$\Sigma^+=\{q\in K | f(q) \geq 0\}$ and $\Sigma^-=\{q \in K |
f(q)\leq 0\}$.\\ 

Designate by $\chi$ the space of $C^1 -$vector fields on $K$ endowed
with the $C^1-$topology. Call \textbf{$\Omega=\Omega(K,f)$} the
space of vector fields $Z: K \rightarrow \R ^{2}$ such that
\begin{equation}\label{eq Z}
 Z(x,y)=\left\{\begin{array}{l} X(x,y),\quad $for$ \quad (x,y) \in
\Sigma^+,\\ Y(x,y),\quad $for$ \quad (x,y) \in \Sigma^-,
\end{array}\right.
\end{equation}
where $X=(f_1,g_1)$, $Y = (f_2,g_2)$ are in $\chi.$ We write
$Z=(X,Y),$ which we will accept to be multivalued in points of
$\Sigma.$ We endow $\Omega$ with the pro\-duct $C^1-$topology. The
trajectories of $Z$ are solutions of  ${\dot q}=Z(q),$ which has, in
general, discontinuous righthand side. The basic results of
differential equations, in this context, were stated by Filippov in
\cite{Fi}.\\ 

\begin{definition} A \textbf{$\mathbf{k}-$parameter family} of
elements in $\Omega$ is a $C^1 -$mapping, with
$r>1$,$$\begin{array}{cccc}
  \zeta: & S^{k} & \longrightarrow & \Omega \\
   & \varrho=(\varrho_1,\varrho_2, \ldots, \varrho_{k}) & \mapsto & X_{\varrho} \\
\end{array}$$
where $S^{k} = [-\epsilon_1,\epsilon_1] \times
[-\epsilon_2,\epsilon_2] \times \ldots [-\epsilon_k,\epsilon_k]$
with $\epsilon_i > 0$, $i=1,2,\ldots,k$, sufficiently small.
\end{definition}

\begin{definition}\label{definicao C0-equiv}
We say that $W, \widetilde{W} \in \chi$ defined in open sets $U$ and
$\widetilde{U}$, respectively, are
$\mathbf{{C}^{0}-}$\textbf{orbitally equivalent} if there exists
an orientation preserving homeomorphism $h: U \rightarrow
\widetilde{U}$ that sends orbits of $W$ to orbits of
$\widetilde{W}$. Here, orbit of $W$ means the image of a solution of
$\dot{x}=W(x)$.
\end{definition}

\begin{definition}\label{definicao sigma-equivalencia}
Two non$-$smooth vector fields $Z=(X,Y), \,
\widetilde{Z}=(\widetilde{X},\widetilde{Y}) \in \Omega(K,f)$ defined
in open sets $U, \, \widetilde{U} \subset K$ and with switching
manifold $\Sigma$ are \textbf{$\mathbf{\Sigma-}$equivalent}  if
there exists an orientation preserving homeomorphism $h: U
\rightarrow \widetilde{U}$ that sends $\Sigma$ in $\Sigma$, the
orbits of  $X$ restrict to $U\cap\Sigma^+$ in the orbits of
$\widetilde{X}$ restrict to $\widetilde{U}\cap\Sigma^+$,  and the
orbits of  $Y$ restrict to $U\cap\Sigma^-$ in the orbits of
$\widetilde{Y}$ restrict to $\widetilde{U}\cap\Sigma^-$.
\end{definition}

Consider the notation \[X.f(p)=\left\langle \nabla f(p),
X(p)\right\rangle \quad \mbox{ and } \quad X^i.f(p)=\left\langle
\nabla X^{i-1}. f(p), X(p)\right\rangle, i\geq 2
\]
where $\langle . , . \rangle$ is the usual inner product in $\R^2$.

\begin{remark}
The vertical dotted lines present in almost all figures of this
paper represent the points $p \in K \subset \R^2$ where $X.f(p)=0$
or $Y.f(p)=0$
\end{remark}

We  distinguish the following regions on the discontinuity set
$\Sigma$:
\begin{itemize}
\item [(i)]$\Sigma^c\subseteq\Sigma$ is the \textit{sewing region} if
$(X.f)(Y.f)>0$ on $\Sigma^c$ .
\item [(ii)]$\Sigma^e\subseteq\Sigma$ is the \textit{escaping region} if
$(X.f)>0$ and $(Y.f)<0$ on $\Sigma^e$.
\item [(iii)]$\Sigma^s\subseteq\Sigma$ is the \textit{sliding region} if
$(X.f)<0$ and $(Y.f)>0$ on $\Sigma^s$.
\end{itemize}

Consider $Z \in \Omega.$ The \textit{sliding vector field}
associated to $Z$ is the vector field  $Z^s$ tangent to $\Sigma^s$
and defined at $q\in \Sigma^s$ by $Z^s(q)=m-q$ with $m$ being the
point of the segment joining $q+X(q)$ and $q+Y(q)$ such that $m-q$
is tangent to $\Sigma^s$ (see Figure \ref{fig def filipov}). It is
clear that if $q\in \Sigma^s$ then $q\in \Sigma^e$ for $-Z$ and then
we  can define the {\it escaping vector field} on $\Sigma^e$
associated to $Z$ by $Z^e=-(-Z)^s$. In what follows we use the
notation $Z^\Sigma$
for both cases.\\

\begin{figure}[!h]
\begin{center}
\psfrag{A}{$q$} \psfrag{B}{$q + Y(q)$} \psfrag{C}{$q + X(q)$}
\psfrag{D}{} \psfrag{E}{\hspace{1cm}$Z^\Sigma(q)$}
\psfrag{F}{\hspace{.7cm}$\Sigma^s$} \psfrag{G}{} \epsfxsize=5.5cm
\epsfbox{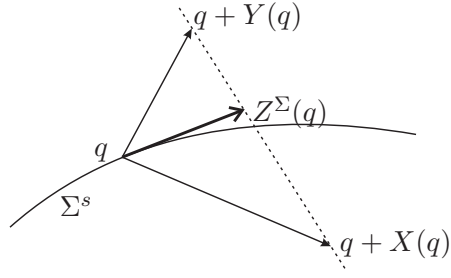} \caption{\small{Filippov's
convention.}} \label{fig def filipov}
\end{center}
\end{figure}

We say that $q\in\Sigma$ is a \textit{$\Sigma-$regular point} if
\begin{itemize}
\item [(i)] $(X.f(q))(Y.f(q))>0$
or
\item [(ii)] $(X.f(q))(Y.f(q))<0$ and $Z^{\Sigma}(q)\neq0$
(that is $q\in\Sigma^e\cup\Sigma^s$ and it is not an equilibrium
point of $Z^{\Sigma}$).\end{itemize}

The points of $\Sigma$ which are not $\Sigma-$regular are called
\textit{$\Sigma-$singular}. We distinguish two subsets in the set of
$\Sigma-$singular points: $\Sigma^t$ and $\Sigma^p$. Any $q \in
\Sigma^p$ is called a \textit{pseudo equilibrium of $Z$} and it is
characterized by $Z^{\Sigma}(q)=0$. Any $q \in \Sigma^t$ is called a
\textit{tangential singularity} and is characterized by
$Z^{\Sigma}(q) \neq 0$ and $(X.f(q))(Y.f(q)) =0$ ($q$ is a contact
point).\\

We say that a point $p_0 \in \Sigma$ is a \textit{$\Sigma-$fold
point} of $X$ if $X.f(p_0)=0$ but $X^{2}.f(p_0)\neq0.$ Moreover,
$p_0\in\Sigma$ is a \textit{visible} (respectively {\it invisible})
$\Sigma-$fold point of $X$
 if $X.f(p_0)=0$ and $X^{2}.f(p_0)> 0$
(respectively $X^{2}.f(p_0)< 0$). We say that a point $q_0 \in
\Sigma$ is a \textit{$\Sigma-$cusp point} of $Y$ if
$Y.f(q_0)=Y^{2}.f(q_0)=0$ and $Y^{3}.f(q_0)\neq0$. Moreover, a
$\Sigma-$cusp point $q_0$ of $Y$ is of \textit{kind 1} (respectively
{\it kind 2})
 if  $Y^{3}.f(q_0)> 0$
(respectively $Y^{3}.f(q_0)< 0$).
In particular, $\Sigma-$fold and $\Sigma-$cusp points are tangential singularities.\\

A pseudo equilibrium $q \in \Sigma^p$ is a \textit{$\Sigma-$saddle}
provided that one of the following conditions is satisfied: (i)
$q\in\Sigma^e$ and $q$ is an attractor for $Z^{\Sigma}$ or (ii)
$q\in\Sigma^s$ and $q$ is a repeller for $Z^{\Sigma}$. A pseudo
equilibrium $q\in\Sigma^p$ is a $\Sigma-$\textit{repeller} (resp.
$\Sigma-$\textit{attractor}) provided $q\in\Sigma^e$ (respectively
$q \in \Sigma^s$) and $q$ is a repeller (respectively, attractor)
equilibrium point for
$Z^{\Sigma}$.\\ 

Given a point $q \in \Sigma^c$, we denote by $r(q)$ the straight
line through $q + X(q)$ and $q + Y(q)$.

\begin{definition}\label{definicao virtual pseudo equilibrio}
The $\Sigma-$regular points $q \in \Sigma^c$ such that either $\{
X(q), Y(q) \}$ is  a li\-nearly dependent set or $r(q) \cap \Sigma =
\emptyset$ are called \textbf{virtual pseudo equilibria}.
\end{definition}

Let us consider a smooth autonomous vector field $W$ defined in an
open set $U$. Then we denote its flow by $\phi_{W}(t,p)$. In this
way,
$$
\left\{
  \begin{array}{ll}
    \dfrac{d}{dt}\phi_{W}(t,p) = W(\phi_{W}(t,p)),\\\\
    \phi_{W}(0,p)=p,
  \end{array}
\right.
$$
where $t \in I= I(p,W)\subset \R$, an interval depending on $p \in
U$ and $W$.

The following definition was stated in \cite{Marcel}, pg 1971.

\begin{definition}\label{definicao trajetorias}
The \textbf{local trajectory} of a NSDS given by \eqref{eq Z} is
defined as follows:
\begin{itemize}
\item For $p \in \Sigma^+$ and $p \in \Sigma^-$ the trajectory is
given by $\phi_{Z}(t,p)=\phi_{X}(t,p)$ and
$\phi_{Z}(t,p)=\phi_{Y}(t,p)$ respectively, where $t \in I$.

\item For $p \in \Sigma^{c}$ such that $X.f(p)>0$, $Y.f(p)>0$ and  taking the
origin of time at $p$, the trajectory is defined as
$\phi_{Z}(t,p)=\phi_{Y}(t,p)$ for $t \in I \cap \{ t \leq 0 \}$ and
$\phi_{Z}(t,p)=\phi_{X}(t,p)$ for $t \in I \cap \{ t \geq 0 \}$. For
the case $X.f(p)<0$ and $Y.f(p)<0$  the definition is the same
reversing time.

\item For $p \in \Sigma^e \cup \Sigma^s$ such that $Z^{\Sigma}(p) \neq 0$ we define
$\phi_{Z}(t,p)=\phi_{Z^{\Sigma}}(t,p)$ for $t\in I$.

\item For $p \in \partial \Sigma^c \cup \partial \Sigma^e \cup
\partial \Sigma^s$ such that the definitions of trajectories  for points
in a full neighborhood of $p$ in $\Sigma$ can be extended to $p$
and coincide, the trajectory through $p$ is this trajectory. 


\item For any other point $\phi_{Z}(t,p)=p$ for all $t \in \R$.
This is the case of points in $ \partial \Sigma^c \cup \partial
\Sigma^e \cup
\partial \Sigma^s$ which are not regular tangential singu\-la\-ri\-ties 
and the equilibrium points of $X$ in $\Sigma_+$, of $Y$ in
$\Sigma_-$ and of $Z^{\Sigma}$ in $\Sigma^s \cup \Sigma^e$.
\end{itemize}
\end{definition}

\begin{definition}\label{definicao orbita}
The \textbf{local orbit$\mathbf{-}$arc} of the vector field $W$
passing through a point $p \in U$ is the set $\gamma_{W} (p) =  \{
\phi_{W}(t,p) : t \in I \}$.
\end{definition}

 Since we are dealing with autonomous systems, from now on
we will use trajectory and orbit$-$arc indistinctly when there is no
danger of confusion.
%

\begin{definition}\label{definicao canard cycles} Consider
 $Z=(X,Y) \in \Omega.$
\begin{enumerate}
\item A  \textbf{canard cycle} is a closed curve $\Gamma= \displaystyle{\bigcup_{i=1}^{n}}\sigma_i$ composed
by the union of orbit$-$arcs $\sigma_i$, $i=1,\hdots,n$, of $X
|_{\Sigma^{+}}$, $Y |_{\Sigma^{-}}$ and $Z^{\Sigma}$ such that:

  \begin{itemize}
  \item Either there exists $i_0 \subset \{ 1, \hdots, n\}$ with $\sigma_{i_{0}}
  \subset \gamma_{X}$ (respectively $\sigma_{i_{0}} \subset \gamma_{Y}$) and then
there exists $j \neq i_0$ with $\sigma_j \subset \gamma_{Y} \cup
\gamma_{Z^{\Sigma}}$ (respectively $\sigma_j \subset \gamma_{X} \cup
\gamma_{Z^{\Sigma}}$), or  $\Gamma$ is composed by a single arc
$\sigma_i$ of $Z^{\Sigma}$;


  \item the transition between arcs of $X$
  and arcs of $Y$ occurs in sewing points;

  \item the transition between arcs of $X$
  (or $Y$) and arcs of $Z^{\Sigma}$ occurs through
  $\Sigma-$fold points or regular points in the escaping or sliding arc, respecting the orientation. Moreover
if $\Gamma\neq\Sigma$ then there exists at least one visible
$\Sigma-$fold point on each connected component of
$\Gamma\cap\Sigma$.
  \end{itemize}

\item A canard cycle $\Gamma$ of $Z$ is
of:

  \begin{itemize}
  \item  \textbf{Kind
  I} if $\Gamma$ meets $\Sigma$ just in sewing points;

  \item \textbf{Kind
  II} if $\Gamma = \Sigma$;

  \item \textbf{Kind
  III} if $\Gamma$ contains at least one visible $\Sigma-$fold point of $Z$.
  \end{itemize}

In Figures \ref{fig canard I}, \ref{fig canard II} and \ref{fig
canard} arise canard cycles of kind I, II and III respectively.

\item A canard cycle $\Gamma$ of $Z$ is \textbf{hyperbolic} if one
of the following conditions are satisfied:

  \begin{itemize}
  \item[(i)] $\Gamma$ is of kind I and $\eta'(p) \neq 1$,
  where $\eta$ is the first return map defined on a segment $T$
  with $p\in T\pitchfork\gamma$;

  \item[(ii)] $\Gamma$ is of kind II;

  \item[(iii)] $\Gamma$ is of kind III,
  $\overline{\Sigma^e} \cap \overline{\Sigma^s} \cap \Gamma = \emptyset$
  and either $\Gamma\cap\Sigma\subseteq\Sigma^c\cup\Sigma^e\cup \Sigma^t$
  or $\Gamma\cap\Sigma\subseteq\Sigma^c\cup\Sigma^s \cup \Sigma^t$.
  \end{itemize}
\end{enumerate}
\end{definition}

\begin{figure}[h!]
\begin{minipage}[b]{0.311\linewidth}
\begin{center}
\epsfxsize=3.3cm \epsfbox{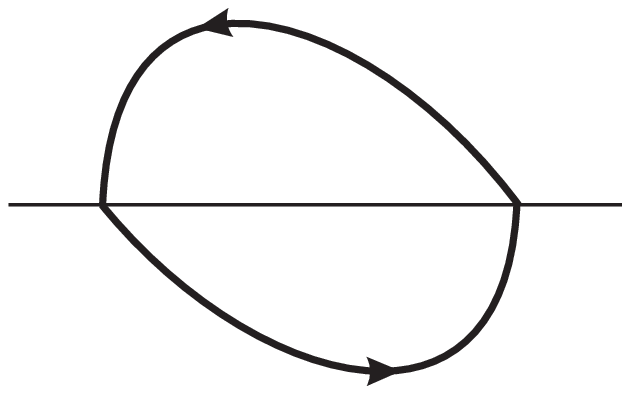} \caption{\small{Canard
cycle of kind I.}} \label{fig canard I}
\end{center}
\end{minipage} \hfill
\begin{minipage}[b]{0.311\linewidth}
\begin{center}
\psfrag{A}{$\Sigma=\Gamma$}\epsfxsize=2.6cm
\epsfbox{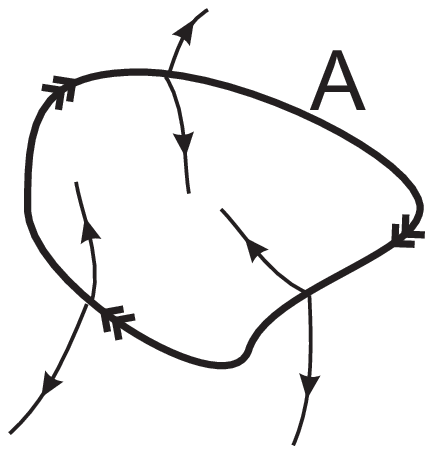} \caption{\small{Canard cycle of kind
II.}}\label{fig canard II}
\end{center}
\end{minipage} \hfill
\begin{minipage}[b]{0.35\linewidth}
\begin{center}
\epsfxsize=4cm  \epsfbox{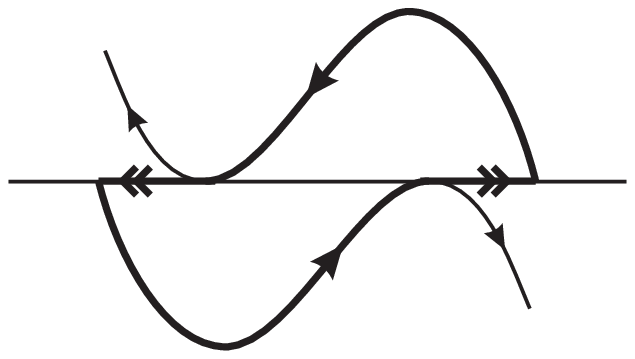}  \caption{\small{Canard
cycle of kind III.}}\label{fig canard}
\end{center}
\end{minipage}
\end{figure}

%

%
%
%
%
%

  Following Theorem $2$ of \cite{vishik}, locally it is
 possible to consider $f(x,y)=y$ and conclude that any
 $X \in \Sigma^{+}$
presenting a $\Sigma-$fold point is $C^0 -$orbitally equivalent to
the normal form $X_0(x,y)= (\rho_1, \rho_2 x)$ with
$\rho_1=\pm1$ and $\rho_2= \pm1$.\\

Following \cite{T1}, we can take $f(x,y)=y$ and derive that any $Y
\in \Sigma^{-}$ presenting a $\Sigma-$cusp point  is
$C^{0}-$orbitally equivalent to the
 normal form $Y_0(x,y)= (\rho_3, \rho_4 x^2)$ with
$\rho_3=\pm1$ and $\rho_4= \pm1$.

\begin{lemma}\label{lema equivalencia} Let $Z=(X,Y) \in \Omega$
presenting a fold$-$cusp singularity,   then $Z$ is
$\Sigma-$equivalent to the standard form
$Z_0=Z_{\rho_1,\rho_2,\rho_3,\rho_4}$ given by
\begin{equation}\label{eq fold-cusp geral}
Z_\rho=Z_{\rho_1,\rho_2,\rho_3,\rho_4} = \left\{
      \begin{array}{ll}
        X_{\rho_1,\rho_2} = \left(
              \begin{array}{c}
                \rho_1  \\
               \rho_2 x
\end{array}
      \right)
 & \hbox{if $y \geq 0$,} \\
        Y_{\rho_3,\rho_4} =  \left(
              \begin{array}{c}
                \rho _3  \\
               \rho_4 x^2
\end{array}
      \right)& \hbox{if $y \leq 0$}
      \end{array}
    \right.
\end{equation}where $\rho_1,\rho_2,\rho_3,\rho_4 = \pm \, 1$.
\end{lemma}

Observe that the values of $\rho_i$, $i=1,2,3,4$, in Lemma \ref{lema
equivalencia} depend on the orientation of $X$ and $Y$. In
Subsection \ref{secao prova lema equivalencia} we prove Lemma
\ref{lema equivalencia}, i.e., we exhibit the homeomorphism that
characterizes the equivalence between any fold$-$cusp singularity
and the standard form given by \eqref{eq fold-cusp geral}.\\

Consider 
 $Z^{ivb,k1}_0,
Z^{vis,k2}_0 \in \Omega$ written in the following standard forms
(similar forms were stated in Section 12 of \cite{Marcel}):
\begin{equation}\label{eq fold-cusp invisivel inicio}
Z^{ivb,k1}_0 = \left\{
      \begin{array}{ll}
        X^{ivb}_0 = \left(
              \begin{array}{c}
                1 \\
               -x
\end{array}
      \right)
 & \hbox{if $y \geq 0$,} \\
        Y^{k1}_0 =  \left(
              \begin{array}{c}
                -1 \\
              -x^2
\end{array}
      \right)& \hbox{if $y \leq 0$, and}
      \end{array}
    \right.
\end{equation}
\begin{equation}\label{eq fold-cusp visivel inicio}
Z^{vis,k2}_0 = \left\{
      \begin{array}{ll}
        X^{vis}_0 = \left(
              \begin{array}{c}
                1 \\
               x
\end{array}
      \right)
 & \hbox{if $y \geq 0$,} \\
        Y^{k2}_0 =  \left(
              \begin{array}{c}
                1 \\
               -x^2
\end{array}
      \right)& \hbox{if $y \leq 0$.}
      \end{array}
    \right.
\end{equation}
Note that $X^{ivb}_0$ presents an invisible $\Sigma-$fold point on
its phase portrait, $X^{vis}_0$ presents a visible $\Sigma-$fold
point, $Y^{k1}_0$ presents a  $\Sigma-$cusp point of kind 1 and
$Y^{k2}_0$ presents a $\Sigma-$cusp point of kind 2. 
 Moreover, in  \eqref{eq fold-cusp invisivel inicio} we made $\rho_1
= 1$ and $\rho_2 = \rho_3 = \rho_4 = -1$ in \eqref{eq fold-cusp
geral} and in
 \eqref{eq fold-cusp visivel inicio} we made $\rho_1 = \rho_2 = \rho_3 = 1$ and $\rho_4 = -1$ in
\eqref{eq fold-cusp geral}. For simplicity we restrict our study to
the normal forms given above, i.e, \eqref{eq fold-cusp invisivel
inicio} and
 \eqref{eq fold-cusp visivel inicio}. All the other choices on the values of $\rho_i$, $i=1,2,3,4$ in \eqref{eq fold-cusp geral} are treated
similarly.

The main problem is to exhibit the bifurcation diagram of
$Z^{\tau,\rho}_0$  where $\tau=ivb$ or
 $vis$ and $\rho=k1$ or $k2$.

%
%
%
%
%
%

In order to detect a larger range of topological behaviors near an
invisible fold$-$cusp singularity we have to refine the analysis
done in \cite{Marcel}. This refinement can be obtained adding a bump
function on the expression of the NSDS.

Denote \[ F(x) = \int g_2(x,\beta,\mu) dx  = \frac{x^3}{3} - \beta x
+ B(x,\beta,\mu) + c_0 , \] where $c_0=- 2 \beta \sqrt{\beta}/3$ and
$g_2$ is the second coordinate of $Y_{\beta,\mu}$ in \eqref{eq
fold-cusp 3 parametros inicio}. The $C^1 -$bump function $B$
satisfies the following properties when $\beta > 0$:

\begin{itemize}

\item It has exactly one point of local minimum in the interval
$(- \sqrt{\beta}, 4\sqrt{\beta})$. This point is located at $x_0 =
\sqrt{\beta}$.

\item $F(3 \sqrt{\beta} + \mu) = 0$ (see Figure \ref{fig bump
function}).  By means of this last property the orbit$-$arc of
$Y_{\beta,\mu}$ that has a quadratic contact to $\Sigma$ at $q_0=(-
\sqrt{\beta},0)$ turns to collide with $\Sigma$ at the point $q_1=(3
\sqrt{\beta} + \mu,0)$. So, the first coordinate of $q_1$ is bigger
(respectively, smaller) than $3 \sqrt{\beta}$ as $\mu$ is bigger
(respectively, smaller) than $0$.

\end{itemize}

\begin{figure}[h!]
\begin{center}\psfrag{A}{$-\sqrt{\beta}$} \psfrag{B}{$0$} \psfrag{C}{$\sqrt{\beta}$}
\psfrag{D}{$$} \psfrag{E}{$3\sqrt{\beta}$}
\psfrag{F}{$4\sqrt{\beta}$}
\psfrag{G}{}
\epsfxsize=4.5cm \epsfbox{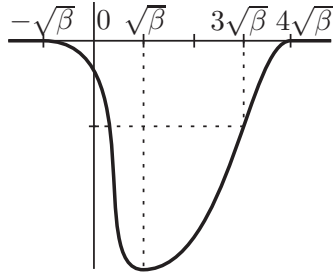}
\caption{\small{Graph of $B$.}} \label{fig bump
function}\end{center}
\end{figure}

\begin{remark}\label{obs parametro mu} It is worth saying that the parameter $\mu$ breaks the strong
proportionality between the roots of $g_2(x, \beta,0)$.  At the
limit value $\mu_0=0$, $Z_{\lambda,\beta,\mu}$ presents distinct
topological behaviors for $\mu < \mu_0$ or $\mu >
\mu_0$.\end{remark}

\begin{remark}\label{obs 1 referee}
Note that in Equations \eqref{eq fold-cusp 3 parametros inicio} and
\eqref{eq fold-cusp 2 parametros inicio} the perturbations
considered depend only on the variable $x$. The local geometry of a
NSDS presenting a cusp$-$fold singularity becomes rather different
if perturbations involving the variables $x$ and
$y$ are admitted.\\
\end{remark}

\subsection{Global Bifurcation}\label{secao global bifurcation}

As said before, the configuration illustrated in Figure \ref{fig
bump com o campo inicio} plays a very important role in our
analysis. The configuration of this figure is reached from \eqref{eq
fold-cusp 3 parametros inicio}, by taking $\beta>0$,
$\lambda=\sqrt{\beta}$ and $\mu=0$. In this section we deal with
this global phenomenon.

Emphasizing, let $Z_0=(X_0,Y_0) \in \Omega$  having the following
properties:

\begin{itemize}

\item The discontinuity set $\Sigma$ is represented by $f(x,y)=y$.

\item Consider $X_0=(f_{1}^{0},g_{1}^{0})$ and $Y_0=(f_{2}^{0},g_{2}^{0})$
and assume that $f_{1}^{0}(p) >0$ if $p \in \Sigma^+$ and
$f_{2}^{0}(p) < 0$ if $p \in \Sigma^-.$

\item $q_0 \in \Sigma$ is a visible $\Sigma-$fold point of $Y_0$ and
$X.f(q_0) > 0$.

\item The orbit $\gamma_{X_0}(q_0)$  of $X_0$ through $q_0$ meets
transversally $\Sigma$ at a point $q_1$.

\item The orbit $\gamma_{Y_0}(q_1)$  of $Y_0$ through $q_1$ meets
tangentially $\Sigma$ at  $q_0$. Call $\Gamma$ the degenerate canard
cycle composed by $\gamma_{X_0}(q_0)$  and $\gamma_{Y_0}(q_1)$. Let
$M$ be the compact region in the plane bounded by $\Gamma.$
\end{itemize}

\subsubsection{Transition Fold Map}\label{secao transition fold
map}

As $q_0 \in \Sigma$ is a visible $\Sigma -$fold point of $Y_0$, we
may assume (see \cite{vishik}) coordinates around $q_0$ such that
the system is represented by $(\dot x,\dot y)=(-1,x)$ with
$q_0=(0,0)$. The solutions of this differential equation are given
by:
$$\phi_{a,b}(t)=(-t + a, -(t^{2}/2) + at + b).$$ The orbit$-$arc
$\phi_0$ through $(0,0)$ is represented by $\phi_{0}(t)=(-t ,
-t^{2}/2)$.

Let $\delta$ be a very small positive number. We construct the
\textit{Transition Map} $\xi: L_{1} \rightarrow L_0$, from $L_1= \{
(x,y), \  y=-\delta, \ x \geq \sqrt{2 \delta} \}$ to $L_0=\{(x,0), x
\geq 0\}$, following the orbits of $Y_0$ (see Figure \ref{fig
primeiro retorno}). The curve $L_1$ is transverse to $Y_0$ at
$p_{\delta}= (  \sqrt{2 \delta} , -\delta).$ Since the solutions
$\phi_{\delta}$ through $(\overline{x},-\delta) \in L_1$ meet
$\Sigma = \{ y=0 \}$ at time $t= \overline{x}\pm \sqrt{
\overline{x}^{\,2} - 2 \delta}$ we obtain that
$\xi(\overline{x})=\sqrt{ \overline{x}^{\,2} - 2 \delta}$ and $\xi$
 is an homeomorphism. Moreover, $\xi^{-1} (x)= \sqrt{x^{2}+ 2 \delta}$, $\xi^{-1}$ is  differentiable  at $0$ and $(\xi^{-1})'(0)=0$.

\begin{figure}[h!]
\begin{center}\psfrag{1}{$L_0$} \psfrag{2}{$L_2$} \psfrag{3}{$L_1$}
\psfrag{4}{\hspace{-1cm}$\gamma_{X}(q_0)$}
\psfrag{5}{$\gamma_{Y}(q_1)$} \psfrag{6}{$q_0$}\psfrag{7}{$p$}
\psfrag{8}{$L_3$}
\psfrag{G}{}
\epsfxsize=4cm \epsfbox{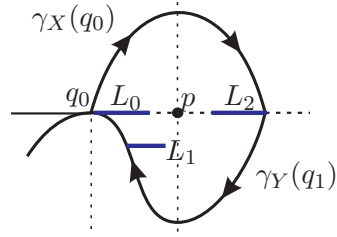}
\caption{\small{First Return Map around the two$-$fold singularity
$p$.}} \label{fig primeiro retorno}\end{center}
\end{figure}

\subsubsection{First Return Map associated to $\Gamma$}\label{secao primeiro retorno em
Gamma}

Let $\varrho_X$ be the transition map from $L_0$ to $L_2 \subset
\Sigma$ via $X_0-$trajectories and $\varrho_Y$ be the transition map
from $L_2$ to $L_1$  via $Y_0-$trajectories (see Figure \ref{fig
primeiro retorno}). Observe that the linear part of the composition
$\varrho_Y \circ \varrho_X$ is nonzero due to the transversality
conditions of the problem. For simplicity, let
$J_{\epsilon}=[0,\epsilon) \times \{ 0 \}$ be a small semi-open
interval of $\Sigma$.

So the First Return Map of $Z_0$ at $q_0$ is $\kappa (x) = (\xi\circ
\varrho_Y \circ \varrho_X)(x)$ for $x \in J_{\epsilon}$.
Its inverse $\kappa^{-1}$ is a differentiable map at $0$ and  satisfies $(\kappa^{-1})' (0) =0.$
So, $\Gamma$ locally repeals the orbits of $Z_0$ closed to $\Gamma$
and in the interior of $\Gamma$.

In conclusion, if $Z$ is very close to $Z_0$ in $\Omega$ in such a
way that it possesses a canard cycle nearby $\Gamma$ then it is a
hyperbolic repeller canard cycle. Under some other  conditions on
$Z_0$ (reversing the directions of $X_0$ and $Y_0$) we can derive
that such canard cycle is an attractor.

\subsubsection{Analysis around the two$-$fold singularity}\label{secao analise na fold-fold
singularity}

 In Equation \eqref{eq fold-cusp 3 parametros
inicio}, for $\beta>0$, it is possible to define a \textit{First
Return Map} $\psi_{\lambda}^{\mu}: (\sqrt{\beta}, 3 \sqrt{\beta} +
\mu) \rightarrow (\sqrt{\beta}, 3 \sqrt{\beta})$, associated to
$Z_{\lambda,\beta,\mu}$, given by
\[
\psi_{\lambda}^{\mu}(x)=(\varrho_{ X_{\lambda}} \circ
\varrho_{Y_{\mu,\beta}})(x)
\]where $\varrho_{Y_{\mu,\beta}}(x)$ is the
first return to $\Sigma$ of the orbit$-$arc of $Y_{\mu,\beta}$ that
passes through $p=(x,0)$ and $\varrho_{X_{\lambda}}(\widetilde{x})$
is the first return to $\Sigma$ of the orbit$-$arc of $X_{\lambda}$
that passes through $p=(\widetilde{x},0)$.

\begin{lemma}\label{lema hipotese H}
If $\beta>0$, $\lambda=\sqrt{\beta}$ and $\mu=0$ in \eqref{eq
fold-cusp 3 parametros inicio} then (see Figure \ref{fig bump com o
campo inicio}) the First Return Map $\psi_{\lambda}^{\mu}(x)$
satisfies
\begin{itemize}
\item[(i)] $\psi_{\lambda}^{0}(x) < x \, , \, \forall x \in (\sqrt{\beta}, 3
\sqrt{\beta})$ and
\item[(ii)] $|(\psi_{\lambda}^{0})'(\sqrt{\beta})| \neq 1$.
\end{itemize}
\end{lemma}
\begin{proof} Consider Figure \ref{fig primeiro retorno}. Given a
point $p \in L_2$, the positive $Y-$orbit by $p$ reaches $L_3$ at
the point $q=(q_1,q_2)$ and the negative $X-$orbit by $p$ reaches
$L_0$ at the point
$\widetilde{p}=(\widetilde{p}_1,\widetilde{p}_2)$. The negative
$Y-$orbit by $\widetilde{p}$ reaches $L_3$ at the point
$\widetilde{q}=(\widetilde{q}_1,\widetilde{q}_2)$. Since $$q_2 -
\widetilde{q}_2 = \displaystyle\frac{ (p_1 - 3 \sqrt{\beta}) (p_1 -
 \sqrt{\beta})^3 (p_1 (1744 + 99 \sqrt{\beta}) -
 \sqrt{\beta} (6256 + 321 \beta))}{384 \beta}$$and $\sqrt{\beta}<p_1<3 \sqrt{\beta}$
 we conclude that $q_2 -
\widetilde{q}_2 > 0$ and item (i) is proved. Item (ii) follows from
Section \ref{secao primeiro retorno em Gamma}.
\end{proof}

 Note that Lemma \ref{lema hipotese H} implies that
$Z_{\sqrt{\beta},\beta,0}$ does not have closed orbits  in the
interior of the closed curve of $Z$ passing through the visible
$\Sigma-$fold point of $Y_{0,\beta}$. Moreover, when $\mu < 0$ (see
Figure \ref{fig dinamica simbolica}), Lemma \ref{lema hipotese H}
guarantees that $\psi_{\lambda}^{\mu}$ has a unique fixed point
$\overline{x}$ where $\overline{x} < 3 \sqrt(\beta) + \mu$. And, in
this case, $|(\psi_{\lambda}^{\mu})'(\overline{x}))| \neq 1$, i.e.,
$\overline{x}$ is a hyperbolic fixed point for
$\psi_{\lambda}^{\mu}$ that corresponds to a hyperbolic canard cycle
of $Z_{\lambda,\beta,\mu}$. When $\mu > 0$ (see Figure \ref{fig
dinamica simbolica}), $\psi_{\lambda}^{\mu}(x) < x$ for all $x \in
(\sqrt{\beta}, 3 \sqrt{\beta} + \mu)$ and closed orbits of
 $Z_{\lambda,\beta,\mu}$ do not arise.\\

\begin{figure}[!h]
\begin{center} \psfrag{A}{$\sqrt{\beta}$} \psfrag{B}{$3 \sqrt{\beta}$}
\psfrag{C}{$3 \sqrt{\beta} + \mu$} \psfrag{D}{$3 \sqrt{\beta}$}
\psfrag{E}{$\mu<0$} \psfrag{F}{$\mu=0$} \psfrag{G}{$\mu>0$}
\psfrag{X}{$\overline{x}$}
 \epsfxsize=11cm \epsfbox{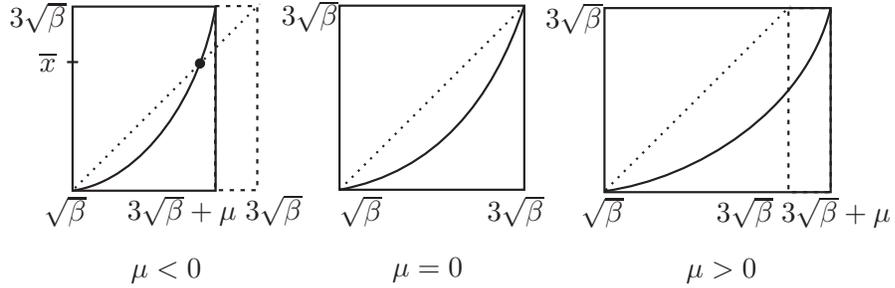}
\caption{\small{Graph of the First Return Map
$\psi_{\lambda}^{\mu}$.}} \label{fig dinamica simbolica}
\end{center}
\end{figure}


Given $Z=(X,Y)$, we describe some properties of both $X =
X_{\lambda}$ and either $Y = Y_{\beta,\mu}$ or $Y = Y_{\beta}$.

The parameter $\lambda$ measures how the $\Sigma-$fold point
$d=(\lambda,0)$ of $X$ is translated away from the origin. More
specifically, if $\lambda < 0$ then $d$ is translated to the left
hand side and if $\lambda
> 0$ then $d$ is
translated to the right hand side.

The parameter $\beta$ distinguishes the contact order between a
trajectory of $Y$ and $\Sigma$. In this way, it occurs one, and only
one, of the following situations:

\begin{itemize}
\item $\mathbf{Y^{+}}$: In this case $\beta > 0$. So $Y$ has two
$\Sigma-$fold points in such a way that one of them invisible and
the other one visible. These points are expressed by $a = a_{\beta}
= (-\sqrt{\beta},0)$ and $b = b_{\beta} = (\sqrt{\beta},0)$.
Moreover, a third point $c=c_{\beta,\mu}=(3 \sqrt{\beta} + \mu,0)$
plays an important role at the analysis of \eqref{eq fold-cusp 3
parametros inicio}. This point is the locus where the orbit$-$arc
$\gamma_{Y}(a)$ intersects transversally $\Sigma$ for negative time
(see Figure \ref{fig duas dobras}). Using the bump function $B$ the
distance between $c$ and $b$ is bigger or smaller than the distance
between $a$ and $b$ according to the value of the parameter $\mu$.
This fact will be important to change from Theorem 1 to Theorems 2
and 3.

\item $\mathbf{Y^{0}}$: In this case $\beta = 0$. So  $Y$ has a  $\Sigma-$cusp point $e = (0,0)$ (see
Figure \ref{fig fold cusps}).

\item $\mathbf{Y^{-}}$: In this case $\beta < 0$. So $Y$ does not have
$\Sigma-$fold points.
In this way, $Y.f \neq 0$ and $Y$ is transversal to $\Sigma$ (see
Figure \ref{fig cusp transversal}).
\end{itemize}

\begin{figure}[!h]
\begin{minipage}[b]{0.493\linewidth}
\begin{center}\psfrag{A}{$b$} \psfrag{B}{$a$} \psfrag{C}{$c$}
\psfrag{D}{\hspace{-.7cm}$\gamma_{Y}(a)$} \psfrag{E}{$X.f=0$}
\psfrag{F}{$\gamma_Y(a)$} \epsfxsize=3cm
\epsfbox{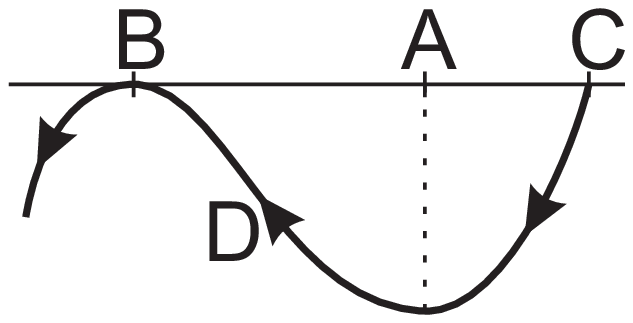} \caption{\small{Case $Y^{+}$.}}
\label{fig duas dobras}\end{center}
\end{minipage} \hfill
\begin{minipage}[b]{0.493\linewidth}
\begin{center}\psfrag{A}{$0$} \psfrag{B}{$X$} \psfrag{D}{$d$}
\epsfxsize=3cm \epsfbox{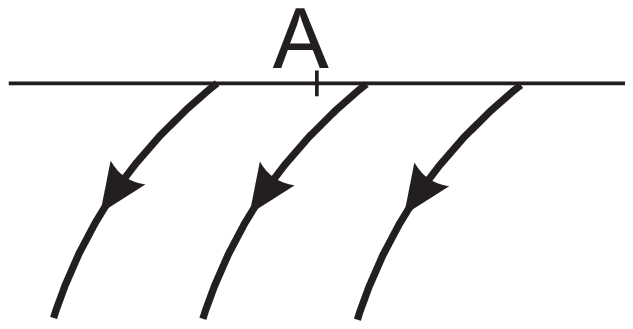}
\caption{\small{Case $Y^{-}$.}} \label{fig cusp transversal}
\end{center}\end{minipage}
\end{figure}

\subsection{The Direction Function}\label{secao funcao direcao}

The next function will be very useful in the sequel.\\

On $\Sigma$, consider  the point $C=(C_{1},C_2)$, the vectors
$X(C)=(D_1,D_2)$ and $Y(C)=(E_1,E_2)$ (as illustrated in Figure
\ref{fig funcao direcao}). Observe that the straight line $r(C)$ by
$q + X(q)$ and $q + Y(q)$, generically, meets $\Sigma$ in a point
$p(C)$. We define the C$^r-$map
$$
\begin{array}{cccc}
  p: & \Sigma & \longrightarrow & \Sigma \\
     & z & \longmapsto & p(z).
\end{array}
$$

We choose local coordinates such that $\Sigma$ is the $x-$axis; so
$C=(C_1,0)$ and $p(C) \in \R \times \{ 0 \}$ can be identified with
points in $\R$. According with this identification, the
\textit{direction function} on $\Sigma$ is defined by
$$
\begin{array}{cccc}
  H: & \R & \longrightarrow & \R \\
     & z & \longmapsto & p(z) - z.
\end{array}
$$
\begin{figure}[!h]
\begin{center}
\psfrag{A}{$A$} \psfrag{B}{$B$} \psfrag{C}{$C$} \psfrag{D}{$\Sigma$}
\psfrag{E}{$X$} \psfrag{F}{$Y$} \psfrag{G}{$C+Y(C)$} \psfrag{H}{$C+
X(C)$} \psfrag{I}{$p(C)$}\psfrag{R}{\hspace{-.5cm}$r(C)$}
\epsfxsize=5cm \epsfbox{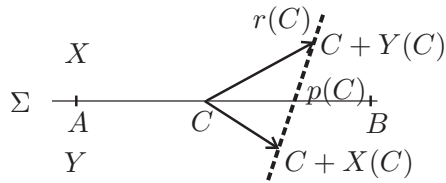}
\caption{\small{Direction function.}} \label{fig funcao direcao}
\end{center}
\end{figure}

\begin{remark}\label{obs sobre zeros H}We obtain that $H$ is a
C$^r-$map. When $C \in \Sigma^e \cup \Sigma^s$ the following holds:

\begin{itemize}
\item if $H(C) < 0$ then the orientation of $Z^{\Sigma}$ in a small neighborhood of $C$ is from $B$ to $A$;

\item if $H(C) = 0$  then  $C \in \Sigma^p$;

\item if $H(C) > 0$ then the orientation of $Z^{\Sigma}$ in a small neighborhood of $C$ is from $A$ to $B$.
\end{itemize}\end{remark}

Simple calculations show that $p(C_1) = \frac{E_2 (D_1+C_1)  - D_2
(E_1+C_1)}{E_2 - D_2}$ and consequently,
\begin{equation}\label{eq H}
H(C_1) = \frac{E_2 D_1  - D_2 E_1}{E_2 - D_2}.
\end{equation}

\begin{remark}\label{obs H positiva ou negativa nas dobras}
If $X.f(p)=0$ and $Y.f(p) \neq 0$ then, in a neighborhood $V_p$ of
$p$ in $\Sigma$, holds $H(V_p) D_1>0$, 
 where $X(p)=(D_1,D_2)$. In fact,
$X.f(p)=0$ and $Y.f(p) \neq 0$ are equivalent to say that $D_2=0$
and $E_2 \neq 0$ in \eqref{eq H}. So,
$\displaystyle\lim_{(D_2,E_2)\rightarrow(0,k_0)} H(p_1) = D_1$,
where $k_0 \neq 0$ and $p=(p_1,p_2)$.
\end{remark}

Considering the previous notation and identifying $\Sigma$ with the
$x-$axis, we have that $r(C) \cap \Sigma = \emptyset$ when
$E_2=D_2$. In such a case, $H$ is not defined at $C$. The following
property is immediate.

\begin{proposition}\label{prop quantidade de equilibria}
If $n_1$ is the number of pseudo equilibria and $n_2$ is the number
of virtual pseudo equilibria then $n_1 + n_2 = v_1 + v_2$ where
$v_1$ is the number of zeros of $H$ and $v_2$ is the number of
points $q$ of $\Sigma$ such that $r(q) \cap \Sigma = \emptyset$.
\end{proposition}
\begin{proof} Straightforward according to Remark \ref{obs sobre zeros H},
Equation \eqref{eq H} and Definition \ref{definicao virtual pseudo
equilibrio}.\\ \end{proof}

\begin{remark}\label{obs parabola na h} Given
$Z_{\lambda,\beta,\mu}$, we list some properties of the function
$H$. According to \eqref{eq H} we have that the expression of $H$ is
\[
H(x, \lambda, \beta ,\mu) = \frac{H_1 (x, \lambda, \beta ,\mu)}{H_2
(x, \lambda, \beta ,\mu)}
\]where $H_1 (x, \lambda, \beta ,\mu) = -x^2 -x + \lambda + \beta - \dfrac{\partial B}{\partial
x}(x,\beta,\mu)$ and $H_2 (x, \lambda, \beta ,\mu) = -x^2 +x -
\lambda + \beta - \dfrac{\partial B}{\partial x}(x,\beta,\mu)$. So,

\begin{itemize}

\item[(i)] When $x=\lambda$ we get $H_1 (\lambda, \lambda, \beta ,\mu) = H_2 (\lambda, \lambda, \beta
,\mu)$.

\item[(ii)] For the parameter values satisfying $\beta= \lambda^2 +
\dfrac{\partial B}{\partial x}(\lambda,\beta,\mu) > 0$ we have $H_1
(\lambda, \lambda, \beta ,\mu)=H_2 (\lambda, \lambda, \beta
,\mu)=0$.

\item[(iii)] Since $H_1(0,0,0,0)=0$ (respectively
$H_2(0,0,0,0)=0$) and \linebreak $\dfrac{\partial H_1}{\partial
x}(0,0,0,0)=-1$ (respectively $\dfrac{\partial H_2}{\partial
x}(0,0,0,0)=1$), by the Implicit Function Theorem there is a unique
$x=x_{H_{1}}(\lambda, \beta ,\mu)$ such that $H_1
(x_{H_{1}}(\lambda, \beta ,\mu), \lambda, \beta ,\mu) = 0$
(respectively $H_2 (x_{H_{2}}(\lambda, \beta ,\mu), \lambda, \beta
,\mu) = 0$). Therefore, there is only one zero of $H_1$ and only one
zero of $H_2$ in a sufficiently small neighborhood of $x=0$. These
points are called $p_1$ and $r_1$, respectively, in Figure \ref{fig
diag bif H}. The pseudo equilibrium $p_1$ and the virtual pseudo
equilibrium $r_1$ are the unique roots of $H_1$ and $H_2$,
respectively, that are relevant to our analysis. In fact, the other
roots are far from the origin.

\end{itemize}
\end{remark}

\begin{figure}[!h]
\begin{center}\psfrag{A}{$r_1$} \psfrag{B}{$r_2$}
\psfrag{C}{$p_1$}\psfrag{D}{$p_2$}\psfrag{E}{$1$}\psfrag{1}{H-1}
\psfrag{2}{H-2} \psfrag{3}{H-3} \psfrag{5}{$\beta$}
\psfrag{4}{$\lambda$} \psfrag{6}{\hspace{-1.4cm}$\beta= \lambda^2 +
\dfrac{\partial B}{\partial x}(\lambda,\beta,\mu)$} \epsfxsize=8cm
\epsfbox{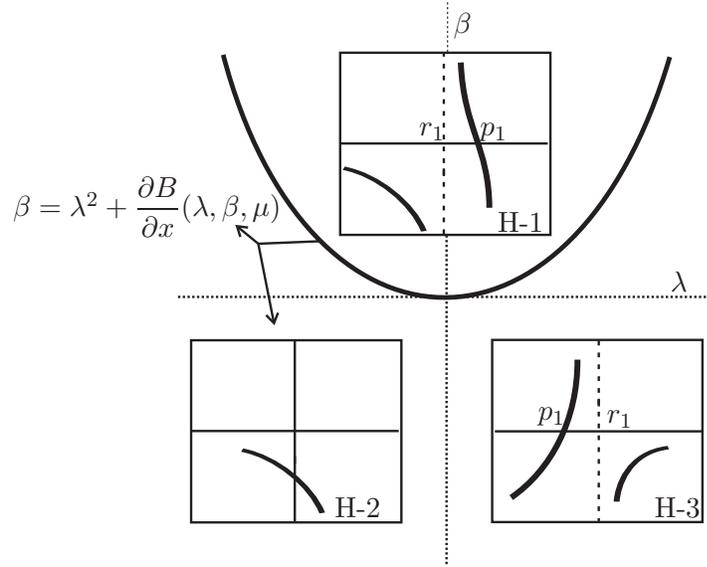} \caption{\small{Variation of $H$
with respect to $\lambda$ and $\beta$. The dark lines in the boxes
H-1, H-2 and H-3 correspond to the graph of $H$.}} \label{fig diag
bif H}
\end{center}
\end{figure}

\subsection{Proof of Lemma \ref{lema equivalencia}}\label{secao prova lema
equivalencia} Now we prove Lemma \ref{lema equivalencia}.
\begin{proof}[Proof of Lemma \ref{lema equivalencia}]
Here we construct a $\Sigma-$preserving homeomorphism $h$ that sends
orbits of $Z=(X,Y)$ to orbits of
$\widetilde{Z}=(\widetilde{X},\widetilde{Y})$, where
$\widetilde{Z}=Z_{\rho}$ is given by \eqref{eq fold-cusp geral} with
$\rho_1=1$ and $\rho_i=-1$, $i=2,3,4$. The other choices on
parameters $\rho_i$, $i=1,2,3,4$, are treated in a similar way. Let
$p$ (respectively, $\widetilde{p}$) be the fold$-$cusp singularity
of $Z$ (respectively, $\widetilde{Z}$) (see Figure \ref{fig lema
equivalencia}). \begin{figure}[!h]
\begin{center}
\psfrag{A}{$\Sigma$} \psfrag{B}{$q_{s}^{1}$} \psfrag{C}{$q$}
\psfrag{D}{$q_{s}^{2}$} \psfrag{E}{$T$} \psfrag{F}{$p_{s}^{1}$}
\psfrag{G}{$\gamma$} \psfrag{H}{$h$} \psfrag{I}{$p_{s}^{2}$}
\psfrag{J}{$p$} \psfrag{K}{$\widetilde{q}_{s}^{1}$}
\psfrag{L}{$\widetilde{\gamma}$} \psfrag{M}{$\widetilde{q}$}
\psfrag{N}{$\widetilde{q}_{s}^{2}$} \psfrag{O}{$\widetilde{T}$}
\psfrag{P}{$\widetilde{p}$} \psfrag{Q}{$\widetilde{p}_{s}^{1}$}
\psfrag{R}{$\widetilde{p}_{s}^{2}$} \psfrag{S}{$\widetilde{\Sigma}$}
\psfrag{U}{$\widetilde{Y}$}
\psfrag{V}{$\widetilde{X}$}\psfrag{X}{$X$} \psfrag{Y}{$Y$}
\epsfxsize=12cm \epsfbox{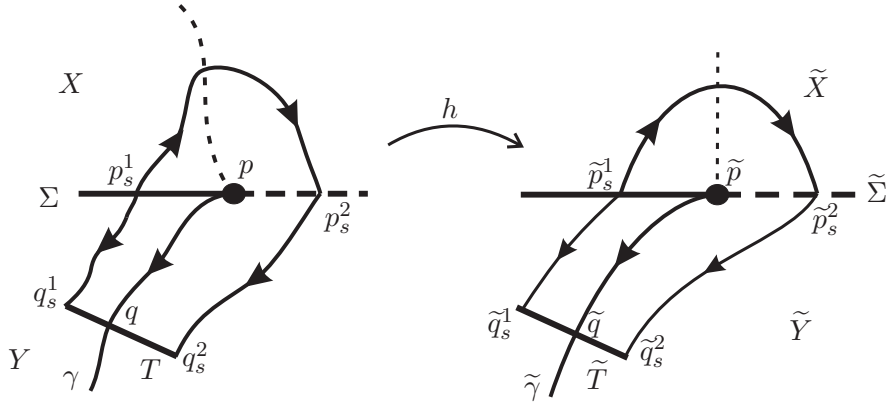}
\caption{\small{Construction of the homeomorphism.}} \label{fig lema
equivalencia}
\end{center}
\end{figure}
Identify $p$ with $\widetilde{p}$, i.e., $h(p)=\widetilde{p}$.
Consider a point $q\in \gamma$ (respectively, $\widetilde{q}\in
\widetilde{\gamma}$), where $\gamma$ (respectively,
$\widetilde{\gamma}$) is the orbit$-$arc of $Y$ (respectively,
$\widetilde{Y}$) starting at $p$ (respectively, $\widetilde{p}$).
Identify $\gamma$ with $\widetilde{\gamma}$ (i.e.,
$h(\gamma)=\widetilde{\gamma}$) from a reparametrization by
arc$-$length. Let $T$ (respectively, $\widetilde{T}$) be transversal
sections to $Y$ (respectively, $\widetilde{Y}$)  passing through $q$
(resp., $\widetilde{q}$) with small amplitude. Identify $T$ with
$\widetilde{T}$ (i.e., $h(T)=\widetilde{T}$) by arc$-$length. Let
$q_{s}^{1} \in T$ be a point on the left of $q$. Using the Implicit
Function Theorem (abbreviated by IFT), there exists a time
$t_{s}^{1} < 0$, depending on $q_{s}^{1}$, such that
$\phi_{Y}(q_{s}^{1},t_{s}^{1}):= p_{s}^{1} \in \Sigma$. Since
$h(T)=\widetilde{T}$, there exists $\widetilde{q}_{s}^{1} \in
\widetilde{T}$ such that $h(q_{s}^{1}) = \widetilde{q}_{s}^{1}$.
Using IFT, there exists a time $\widetilde{t}_{s}^{1} < 0$,
depending on $\widetilde{q}_{s}^{1}$, such that
$\phi_{\widetilde{Y}}(\widetilde{q}_{s}^{1},\widetilde{t}_{s}^{1}):=
\widetilde{p}_{s}^{1} \in \widetilde{\Sigma}$. Identify the
orbit$-$arc $\sigma^{p_{s}^{1}}_{q_{s}^{1}}(Y)$ of $Y$ joining
$p_{s}^{1}$ to $q_{s}^{1}$ with the orbit$-$arc
$\widetilde{\sigma}^{\widetilde{p_{s}^{1}}}_{\widetilde{q_{s}^{1}}}(\widetilde{Y})$
of $\widetilde{Y}$ joining $\widetilde{p}_{s}^{1}$ to
$\widetilde{q}_{s}^{1}$ (i.e.,
$h(\sigma^{p_{s}^{1}}_{q_{s}^{1}}(Y))=\widetilde{\sigma}^{\widetilde{p_{s}^{1}}}_{\widetilde{q_{s}^{1}}}(\widetilde{Y})$)
by arc$-$length. Fix the notation for the orbit$-$arcs of a given
vector field joining two points. Since $p$ (respectively,
$\widetilde{p}$) is a $\Sigma-$fold point of $X$ (respectively,
$\widetilde{X}$), using the IFT, there exists a time $t_{s}^{2} > 0$
(respectively, $\widetilde{t}_{s}^{2} > 0$), depending on
$p_{s}^{1}$ (respectively, $\widetilde{p}_{s}^{1}$), such that
$\phi_{X}(p_{s}^{1},t_{s}^{2}):= p_{s}^{2} \in \Sigma$
(respectively,
$\phi_{\widetilde{X}}(\widetilde{p}_{s}^{1},\widetilde{t}_{s}^{2}):=
\widetilde{p}_{s}^{2} \in \widetilde{\Sigma}$). Identify
$\sigma^{p_{s}^{2}}_{p_{s}^{1}}(X)$ with
$\widetilde{\sigma}^{\widetilde{p_{s}^{2}}}_{\widetilde{p_{s}^{1}}}(\widetilde{X})$
(i.e.,
$h(\sigma^{p_{s}^{2}}_{p_{s}^{1}}(X))=\widetilde{\sigma}^{\widetilde{p_{s}^{2}}}_{\widetilde{p_{s}^{1}}}(\widetilde{X})$)
by arc$-$length. Using the IFT,  there exists a time $t_{s}^{3} > 0$
(respectively, $\widetilde{t}_{s}^{3} > 0$), depending on
$p_{s}^{2}$ (respectively, $\widetilde{p}_{s}^{2}$), such that
$\phi_{Y}(p_{s}^{2},t_{s}^{3}):= q_{s}^{2} \in T$ (resp.,
$\phi_{\widetilde{Y}}(\widetilde{p}_{s}^{2},\widetilde{t}_{s}^{3}):=
\widetilde{q}_{s}^{2} \in \widetilde{T}$). Identify
$\sigma^{q_{s}^{2}}_{p_{s}^{2}}(Y)$ with
$\widetilde{\sigma}^{\widetilde{q_{s}^{2}}}_{\widetilde{p_{s}^{2}}}(\widetilde{Y})$
(i.e.,
$h(\sigma^{q_{s}^{2}}_{p_{s}^{2}}(Y))=\widetilde{\sigma}^{\widetilde{q_{s}^{2}}}_{\widetilde{p_{s}^{2}}}(\widetilde{Y})$)
by arc$-$length.

So, the homeomorphism $h$ sends $\Sigma$ to $\widetilde{\Sigma}$ and
sends orbits of $Z$ to orbits of $\widetilde{Z}$. 

\end{proof}


\section{Proof of Theorem 1}\label{secao prova teorema 1}

\begin{proof}[Proof of Theorem 1] In Case $1_1$ we assume
that $Y$ presents the behavior $Y^{-}$ where $\beta<0$. In Cases
$2_1$, $3_1$ and $4_1$ we assume that $Y$ presents the behavior
$Y^0$ where $\beta=0$. In these cases canard cycles do not arise
(for a proof, see \cite{Eu-canard-cycles}).

$\diamond$ \textit{Case $1_1$. $\beta<0$:} The points of $\Sigma$ on
the left of $d$ belong to $\Sigma^e$ and the points on the right of
$d$ belong to $\Sigma^c$. See Figure \ref{fig 1 teo 1}. Since
$\beta<0$, the graph of $H$ is illustrated in H-3 of Figure \ref{fig
diag bif H}. We get that $p_1 = (-1+ \sqrt{1+ 4 \beta + 4
\lambda}/2,0) \in \Sigma^e$ is a $\Sigma-$repeller and  $r_1 = (1
-\sqrt{1+ 4 \beta - 4 \lambda}/2,0) \in \Sigma^c$.

\begin{figure}[!h]
\begin{center}\psfrag{A}{$p_1$} \psfrag{B}{$p_1$} \psfrag{C}{$r_1$}\psfrag{D}{$r_2$}
\epsfxsize=3.5cm \epsfbox{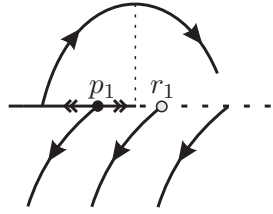}
\caption{\small{Case $1_1$.}} \label{fig 1 teo 1}
\end{center}
\end{figure}


$\diamond$ \textit{Case $2_1$. $\lambda<0$, Case $3_1$. $\lambda=0$
and Case $4_1$. $\lambda>0$:} The configuration of the connected
components of $\Sigma$ is the same as Case $1_1$. Since $\beta=0$,
the graph of $H$, when $\lambda\neq0$, is given by H-3 of Figure
\ref{fig diag bif H}. When $\lambda=0$ (Case $3_1$), the graph of
$H$ is given by H-2 of Figure \ref{fig diag bif H} and $p_1 = r_1$.
These cases are illustrated in Figure \ref{fig teo1 casos 2 3 4}.

\begin{figure}[!h]
\begin{center}\psfrag{A}{$p_1$} \psfrag{B}{$p_1$} \psfrag{C}{$r_1$}
\psfrag{D}{$r_2$}\psfrag{E}{$e$} \psfrag{F}{$d<e$} \psfrag{G}{$d=e$}
\psfrag{H}{$d>e$} \psfrag{I}{$2_1$} \psfrag{J}{$3_1$}
\psfrag{K}{$4_1$} \epsfxsize=12.2cm
\epsfbox{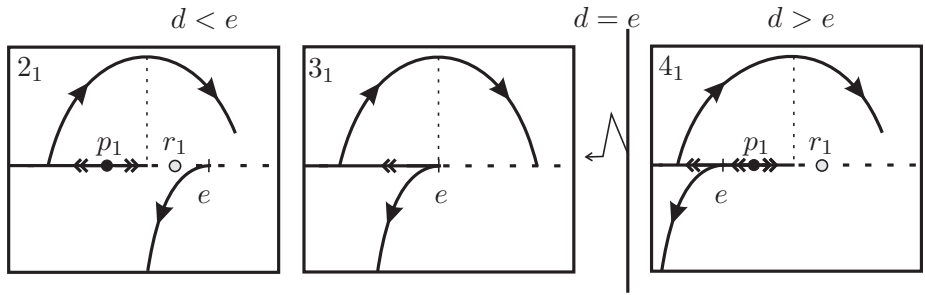} \caption{\small{Cases $2_1$,
$3_1$ and $4_1$.}} \label{fig teo1 casos 2 3 4}
\end{center}
\end{figure}

In Cases $5_1 - 17_1$ we assume that $Y$ presents the behavior $Y^+$
where $\beta>0$.

$\diamond$ \textit{Case $5_1$. $\lambda< - \sqrt{\beta}$:} The
points of $\Sigma$ on the left of $d$ belong to $\Sigma^e$, the
points inside the interval $(a,b)$ belong to $\Sigma^s$ and the
points on $(d,a)$ and on the right of $b$ belong to $\Sigma^c$. The
graph of $H$ is like H-3 of Figure \ref{fig diag bif H}. We can
prove that $p_1$ is a $\Sigma-$repeller  situated on the left
of $d$ and  $r_1 \in (d,a)$. 
canard cycles do not arise. See Figure \ref{fig teo 1 caso 5}.

\begin{figure}[!h]
\begin{minipage}[b]{0.493\linewidth}
\begin{center}\psfrag{A}{$a$} \psfrag{B}{$b$} \psfrag{C}{$c$}
\psfrag{D}{\hspace{-.2cm}$\gamma_1$} \psfrag{E}{$X.f=0$}
\psfrag{F}{$\gamma_1$} \epsfxsize=3.5cm
\epsfbox{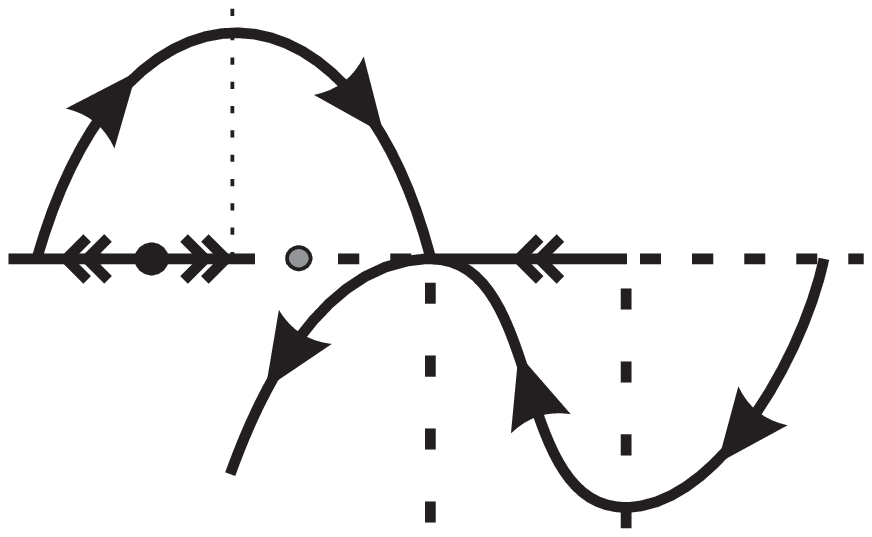} \caption{\small{Case $5_1$.}}
\label{fig teo 1 caso 5}\end{center}
\end{minipage} \hfill
\begin{minipage}[b]{0.493\linewidth}
\begin{center}\psfrag{A}{$0$} \psfrag{B}{$X$} \psfrag{D}{$d$}
\epsfxsize=3cm \epsfbox{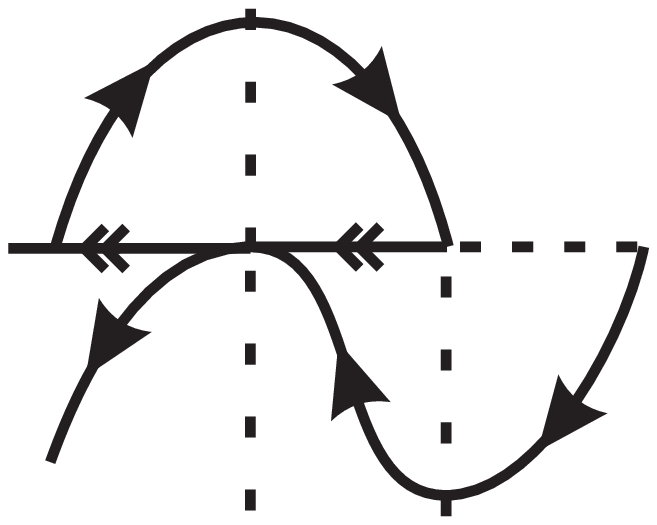}
\caption{\small{Case $6_1$.}} \label{fig teo 1 caso 6}
\end{center}\end{minipage}
\end{figure}

$\diamond$ \textit{Case $6_1$. $\lambda = - \sqrt{\beta}$:} In this
case the points on the right of $b$ belong to $\Sigma^c$, the points
on $(a=d,b)$ belong to $\Sigma^s$ and the points on the left of
$a=d$ belong to $\Sigma^e$. Since $\beta=\lambda^2$, $H$ is like H-2
of Figure \ref{fig diag bif H} and $p_1 = r_1$. 
There exists a non hyperbolic canard cycle $\Gamma$ of kind III
passing through $a$ and $c$. See Figure \ref{fig teo 1 caso 6}.

$\diamond$ \textit{Case $7_1$. $- \sqrt{\beta} < \lambda < 0$, Case
$8_1$. $\lambda = 0$ and Case $9_1$. $0 < \lambda < \sqrt{\beta} $:}
 The configuration of the connected components of $\Sigma$ is like Case $5_1$ replacing $a$ by $d$ and vice-versa. The
graph of $H$ is like H-1 of Figure \ref{fig diag bif H}. We observe
that $p_1 \in (d,b)$ is a $\Sigma-$attractor and $r_1 \in (a,d)$. %
There exists a hyperbolic repeller canard cycle
 $\Gamma$ of kind III passing through $a$ and $c$.
See Figure \ref{fig teo 1 casos 7 8 9}.

\begin{figure}[!h]
\begin{center}\psfrag{a}{$a$} \psfrag{b}{$b$} \psfrag{c}{$c$} \psfrag{d}{$d$}  \psfrag{b=d}{$b=d$} \psfrag{1}{$7_1$} \psfrag{2}{$8_1$} \psfrag{3}{$9_1$} \psfrag{4}{$b<d<0$}
\psfrag{5}{$d=0$}\psfrag{6}{$0<d<c$} \epsfxsize=12.2cm
\epsfbox{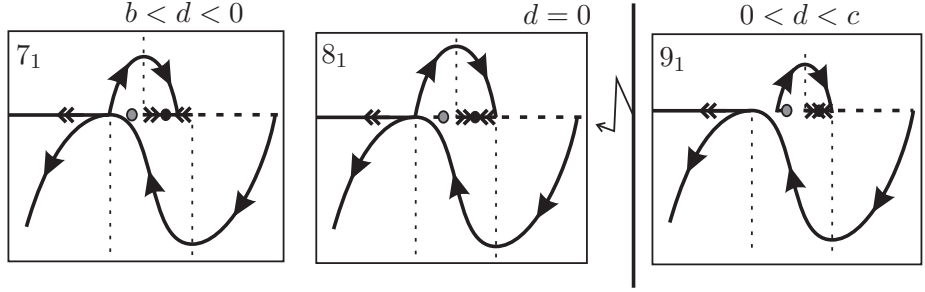} \caption{\small{Cases $7_1 -
9_1$.}} \label{fig teo 1 casos 7 8 9}
\end{center}
\end{figure}

$\diamond$ \textit{Case $10_1$. $\lambda = \sqrt{\beta}$:}  In this
case the points on the left of $a$ belong to $\Sigma^e$ and the
points  on the right of $a$ belong to $\Sigma^c$, except by $Q=(b,0)
\in \Sigma$. Since $\beta=\lambda^2$, $H$ is like H-2 of Figure
\ref{fig diag bif H} and $p_1 = r_1$. 
Since $\mu = 0$ and $d=b$, by the construction of the bump function
$B$ it is straightforward to show that the point $Q$ behaves itself
like a weak attractor for $Z$ and there exists a non hyperbolic
canard cycle of kind III passing through $a$ and $c$. See Figure
\ref{fig bump com o campo inicio}. This case has already been
discussed previously in Subsection \ref{secao global bifurcation}.
Note that in \cite{Marcel} the authors avoid this case.

$\diamond$ \textit{Case $11_1$. $\sqrt{\beta}<\lambda< L_1 $:} The
meaning of the value $L_1$ will be given below in this case. The
points of $\Sigma$ on the left of $a$ and on $(b,d)$ belong to
$\Sigma^e$. The points on $(a,b)$ and on the right of $d$ belong to
$\Sigma^c$. The graph of $H$ is like H-3 of Figure \ref{fig diag bif
H}. We can prove that $p_1 \in (b,d)$ is a $\Sigma-$repeller and
$r_1$ is on the right of $d$. Since the point $Q$ of the previous
case is a weak attractor, in a neighborhood of $d$ occurs a
\textit{Like Hopf Bifurcation}. Moreover, according to Lemma
\ref{lema hipotese H}, there is a unique canard cycle $\Gamma_1$ in
a neighborhood of $d$ and a unique canard cycle $\Gamma_2$ in a
neighborhood of $c$. Observe that both are of kind I, $\Gamma_1$ is
attractor, $\Gamma_2$ is repeller and $\Gamma_1$ is located within
the region bounded by $\Gamma_2$. See Figure \ref{fig nova 2
ciclos}. Note that, as $\lambda$ increases, $\Gamma_1$ becomes
bigger and $\Gamma_2$ becomes smaller. When $\lambda$ assumes the
limit value $L_1$, one of them collides with the other.

\begin{figure}[!h]
\begin{center}\psfrag{a}{$a$} \psfrag{b}{$b$} \psfrag{c}{$c$} \psfrag{d}{$d$}  \psfrag{b=d}{$b=d$}\psfrag{A}{$\lambda =L_1$} \psfrag{B}{$\sqrt{\beta}<\lambda < L_1$} \psfrag{C}{$\lambda = L_1$} \psfrag{D}{$ \lambda = -\frac{\beta}{2}$}
\psfrag{E}{$-\frac{\beta}{2} < \lambda < 0$}\psfrag{0}{$12_1$}
\psfrag{1}{$11_1$} \psfrag{2}{$12_1$}
\psfrag{3}{$10_1$}\psfrag{4}{$11_1$} \epsfxsize=9cm
\epsfbox{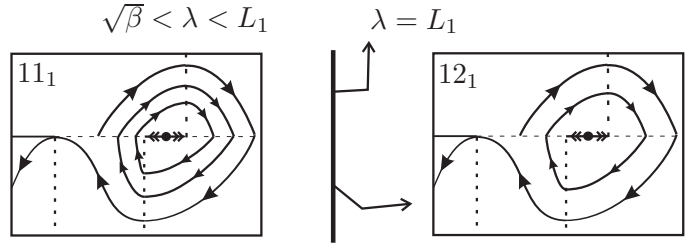} \caption{\small{Cases $11_1$ and
$12_1$.}} \label{fig nova 2 ciclos}
\end{center}
\end{figure}

$\diamond$ \textit{Case $12_1$. $\lambda = L_1$:} The distribution
of the connected components of $\Sigma$ and the behavior of $H$ are
the same as Case $11_1$. Since $\lambda = L_1$, as described in the
previous case, there exists a non hyperbolic canard cycle $\Gamma$
of kind I which is an attractor for the trajectories inside it and
is a repeller for the trajectories outside it. See Figure \ref{fig
nova 2 ciclos}.

$\diamond$ \textit{Case $13_1$. $L_1<\lambda< 2 \sqrt{\beta} $, Case
$14_1$. $\lambda = 2 \sqrt{\beta}$, Case $15_1$. $2 \sqrt{\beta} <
\lambda < 3 \sqrt{\beta}$, Case $16_1$. $\lambda = 3 \sqrt{\beta}$
and Case $17_1$. $\lambda
> 3 \sqrt{\beta}$:}  The distribution
of the connected components of $\Sigma$ and the behavior of $H$ are
the same as Case $11_1$. Canard cycles do not arise. See Figure
\ref{fig teo 1 casos 11 12 13 14 15}.

\begin{figure}[!h]
\begin{center}\psfrag{a}{$a$} \psfrag{b}{$b$} \psfrag{c}{$c$} \psfrag{d}{$d$}  \psfrag{b=d}{$b=d$}\psfrag{A}{\hspace{-.7cm}$0<\lambda < \beta / 2$} \psfrag{B}{$\lambda= \beta / 2$} \psfrag{C}{\hspace{.2cm}$\beta / 2 < \lambda
< \beta$} \psfrag{D}{$ \lambda = \beta$} \psfrag{E}{$\lambda >
\beta$}\psfrag{1}{$13_1$}  \psfrag{2}{$14_1$}
\psfrag{3}{$15_1$}\psfrag{4}{$16_1$} \psfrag{5}{$17_1$}
\epsfxsize=11.5cm \epsfbox{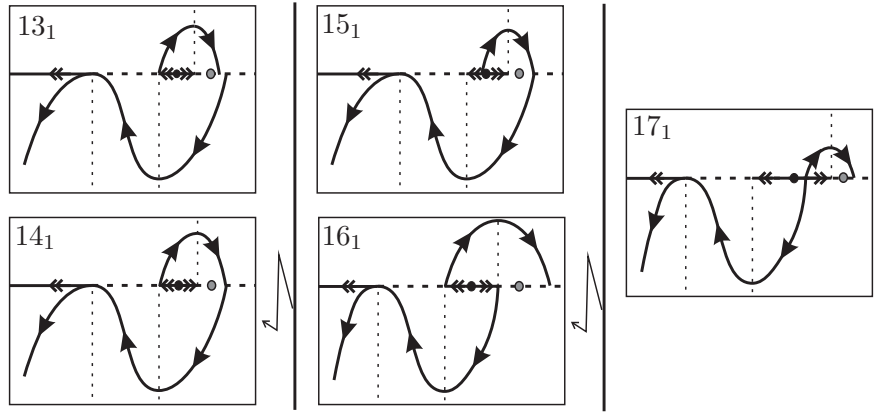}
\caption{\small{Cases $13_1 - 17_1$.}} \label{fig teo 1 casos 11 12
13 14 15}
\end{center}
\end{figure}

\begin{figure}[!h]
\begin{center}\psfrag{A}{$11_1$} \psfrag{B}{$12_1$} \psfrag{C}{$13_1$} \psfrag{D}{$14_1$} \psfrag{E}{$15_1$}
\psfrag{F}{$\lambda$}  \psfrag{G}{$\beta$}
\psfrag{H}{$\lambda=-\sqrt{\beta}$} \psfrag{I}{$\lambda=
2\sqrt{\beta}$} \psfrag{J}{$\lambda= 3\sqrt{\beta}$}
\psfrag{K}{$\lambda=\sqrt{\beta}$}
 \psfrag{1}{$1_1$} \psfrag{2}{$2_1$} \psfrag{3}{$3_1$}
  \psfrag{4}{$4_1$} \psfrag{5}{$5_1$} \psfrag{6}{$6_1$}
   \psfrag{7}{$7_1$} \psfrag{8}{$8_1$} \psfrag{9}{$9_1$}
    \psfrag{0}{$10_1$}\psfrag{Y}{$16_1$}\psfrag{Z}{$17_1$}\psfrag{W}{$L_1$}
\epsfxsize=8cm \epsfbox{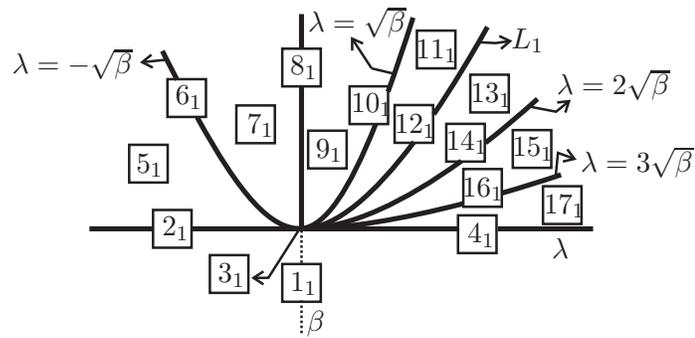}
\caption{\small{Bifurcation Diagram of Theorem 1.}} \label{fig diag
bif teo 1}
\end{center}
\end{figure}

%
The Bifurcation Diagram is illustrated in Figure \ref{fig diag bif
teo 1}.\end{proof}

\begin{remark}\label{observacao separatrizes} In Cases $9_1$ and $11_1$
the ST$-$bifurcations (as described in \cite{Marcel}, Subsections
11.2 and 12.2) arise. In fact, note that the trajectory passing
through $a$, in Case $9_1$, and $c$, in Case $11_1$, can make more
and more turns around $p_1$. This fact characterizes a global
bifurcation also reached in other cases.
\end{remark}


\section{Proof of Theorem 2}\label{secao prova teorema 2}

\begin{proof}[Proof of Theorem 2] In Case $1_2$ we assume
that $Y$ presents the behavior $Y^-$.  In Cases $2_2$, $3_2$ and
$4_2$ we assume that $Y$ presents the behavior $Y^0$. In Cases $5_2
- 19_2$ we assume that $Y$ presents the behavior $Y^+$.

$\diamond$ \textit{Case $1_2$. $\beta<0$, Case $2_2$. $\lambda<0$,
Case $3_2$. $\lambda=0$, Case $4_2$. $\lambda>0$. Case $5_2$.
$\lambda< - \sqrt{\beta}$. Case $6_2$. $\lambda = - \sqrt{\beta}$.
Case $7_2$. $- \sqrt{\beta} < \lambda < 0$ and Case $8_2$. $\lambda
= 0$: } By the choice of the bump function $B$, these cases are
analogous to Cases $1_1$, $2_1$, $3_1$, $4_1$, $5_1$, $6_1$, $7_1$
and $8_1$.


$\diamond$ \textit{Case $9_2$. $0 < \lambda < \sqrt{\beta} -
\mu/2$:} The analysis of this case is done in a similar way as the
Case $9_1$. In this case and in Cases $7_2$ and $8_2$ there exists a
hyperbolic repeller canard cycle $\Gamma$ of kind III passing
through $a$ and $c$.

$\diamond$ \textit{Case $10_2$. $\lambda = \sqrt{\beta} - \mu/2$:}
The points of $\Sigma$ on the left of $a$ belong to $\Sigma^e$ and
the points on $(d,b)$ belong to $\Sigma^s$. The points on $(a,d)$
and on the right of $b$ belong to $\Sigma^c$. The graph of $H$ is
like H-3 of Figure \ref{fig diag bif H}. Observe  that $p_1 \in
(d,b)$ is a $\Sigma-$attractor and  $r_1$ is on the right of $b$. In
this case the arc $\gamma_{X}(a)$ of $X$ passing through $a$ returns
to $\Sigma$ at the point $c$. So, in this case there arises a non
hyperbolic canard cycle $\Gamma = \gamma_{X}(a) \cup \gamma_{Y}(c)$.
By the discussion on subsection \ref{secao primeiro retorno em
Gamma},  we have that $\Gamma$ is a repeller and we do not have
other canard cycles inside $\Gamma$. See Figure \ref{fig teo 2 casos
10 11 12}.

\begin{figure}[!h]
\begin{center}\psfrag{a}{$a$} \psfrag{b}{$b$} \psfrag{c}{$c$} \psfrag{d}{$d$}  \psfrag{b=d}{$b=d$}\psfrag{A}{\hspace{-.7cm}$0<\lambda < \beta / 2$} \psfrag{B}{$\lambda= \beta / 2$} \psfrag{C}{\hspace{.2cm}$\beta / 2 < \lambda
< \beta$} \psfrag{D}{$ \lambda = \beta$} \psfrag{E}{$\lambda >
\beta$}\psfrag{1}{$10_2$}  \psfrag{2}{$11_2$}
\psfrag{3}{$12_2$}\psfrag{4}{$14_1$} \psfrag{5}{$15_1$}
\epsfxsize=12cm \epsfbox{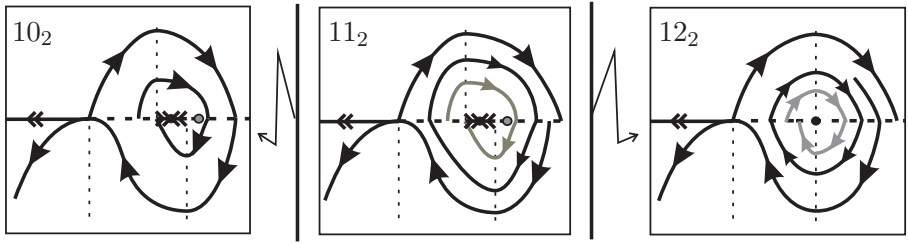}
\caption{\small{Cases $10_2 - 12_2$.}} \label{fig teo 2 casos 10 11
12}
\end{center}
\end{figure}

$\diamond$ \textit{Case $11_2$. $\sqrt{\beta} - \mu/2 <\lambda<
\sqrt{\beta}$:}  The configuration on $\Sigma$ and the graph of $H$
are the same as Case $10_2$. Since $\varrho_{X}^{-1}(c) \in (a,d)$
there exists a point $Q \in
(\varrho_{X}^{-1}(c),\varrho_{X}^{-1}(b))$ such that $\eta'(Q)=1$.
So there exists a hyperbolic repeller canard cycle $\Gamma$, of kind
I, passing through $Q$. See Figure \ref{fig teo 2 casos 10 11 12}.
Moreover, by Lemma \ref{lema hipotese H}  this canard cycle is
unique. In Figure \ref{fig dinamica simbolica} we introduce the
point $\overline{x}$ which plays the same role of $Q$.

$\diamond$ \textit{Case $12_2$. $\lambda = \sqrt{\beta}$:}  The
points of $\Sigma$ on the left of $a$ belong to $\Sigma^e$ and the
points on the right of $a$ belong to $\Sigma^c$, except by the
tangential singularity $c=d$. The graph of $H$ is like H-2 of Figure
\ref{fig diag bif H}. The repeller canard cycle $\Gamma$ presented
in the previous case is persistent. Recall that this canard cycle is
born from the bifurcation of Case $10_2$. So, the radius of $\Gamma$
does not tend to zero when $\lambda$ tends to $\sqrt{\beta}$.
Moreover, the tangential singularity $b=d$ behaves itself like a
weak attractor. See Figure \ref{fig teo 2 casos 10 11 12}.

$\diamond$ \textit{Case $13_2$. $\sqrt{\beta} <\lambda< L_1$,  Case
$14_2$. $\lambda = L_1$,  Case $15_2$. $L_1 < \lambda < 2
\sqrt{\beta} + \mu/2 $, Case $16_2$. $\lambda = 2 \sqrt{\beta} +
\mu/2$, Case $17_2$. $2 \sqrt{\beta} + \mu/2 < \lambda < 3
\sqrt{\beta} + \mu$, Case $18_2$. $\lambda = 3 \sqrt{\beta} + \mu$
and Case $19_2$. $\lambda
> 3 \sqrt{\beta} + \mu$:} The analysis of
these cases is done in a  similar way as Cases $11_1$, $12_1$,
$13_1$, $14_1$, $15_1$, $16_1$ and $17_1$, respectively.

\begin{figure}[!h]
\begin{center}\psfrag{A}{$13_2$} \psfrag{B}{$14_2$} \psfrag{C}{$15_2$} \psfrag{D}{$16_2$} \psfrag{E}{$17_2$}
\psfrag{F}{$\lambda$}  \psfrag{G}{$\beta$}
\psfrag{H}{$\lambda=-\sqrt{\beta}$} \psfrag{I}{$\lambda=
2\sqrt{\beta}$} \psfrag{J}{$\lambda= 3\sqrt{\beta}$}
\psfrag{K}{$\lambda=\sqrt{\beta}$}
\psfrag{L}{$10_2$}\psfrag{M}{$11_2$} 
 \psfrag{1}{$1_2$} \psfrag{2}{$2_2$} \psfrag{3}{$3_2$}
  \psfrag{4}{$4_2$} \psfrag{5}{$5_2$} \psfrag{6}{$6_2$}
   \psfrag{7}{$7_2$} \psfrag{8}{$8_2$} \psfrag{9}{$9_2$}
    \psfrag{0}{$12_2$}\psfrag{Y}{$18_2$}\psfrag{Z}{$19_2$}
\epsfxsize=8cm \epsfbox{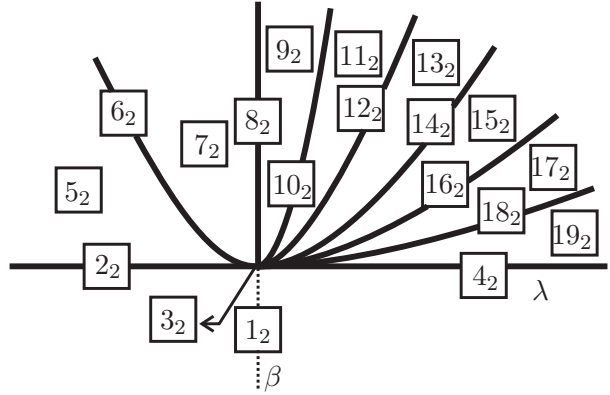}
\caption{\small{Bifurcation Diagram of Theorems 2 and 3.}}
\label{fig diag bif teo 2}
\end{center}
\end{figure}

The bifurcation diagram is illustrated in Figure \ref{fig diag bif
teo 2}.\end{proof}


\section{Proof of Theorem 3}\label{secao prova teorema 3}

\begin{proof}[Proof of Theorem 3] In Case $1_3$ we assume
that $Y$ presents the behavior $Y^-$.  In Cases $2_3$, $3_3$ and
$4_3$ we assume that $Y$ presents the behavior $Y^0$. In Cases $5_3
- 19_3$ we assume that $Y$ presents the behavior $Y^+$.

$\diamond$ \textit{Case $1_3$. $\beta<0$, Case $2_3$. $\lambda<0$,
Case $3_3$. $\lambda=0$, Case $4_3$. $\lambda>0$, Case $5_3$.
$\lambda< - \sqrt{\beta}$, Case $6_3$. $\lambda = - \sqrt{\beta}$,
Case $7_3$. $- \sqrt{\beta} < \lambda < 0$, Case $8_3$. $\lambda =
0$ and Case $9_3$. $0 < \lambda < \sqrt{\beta}$:} By the choice of
the bump function $B$, these cases are analogous to Cases $1_1$,
$2_1$, $3_1$, $4_1$, $5_1$, $6_1$, $7_1$, $8_1$ and $9_1$.


$\diamond$ \textit{Case $10_3$. $\lambda = \sqrt{\beta}$:} The
distribution of the connected components of $\Sigma$ and the
behavior of $H$ are the same as Case $12_2$. This case differs from
Case $12_2$ because, as observed in Subsection \ref{secao analise na
fold-fold singularity}, when $\lambda = \sqrt{\beta}$ and $\mu>0$
canard cycles of  $Z$ do not arise (see Figure \ref{fig dinamica
simbolica}) bifurcating from the non hyperbolic canard cycle
$\Gamma$ of Case $12_3$ below. Moreover, the tangential singularity
$d=b$ behaves itself like a weak attractor. See Figure \ref{fig teo
3 casos 10 11 12}. There exists a hyperbolic repeller canard cycle
 $\Gamma$ of kind III passing through $a$ and $c$.

\begin{figure}[!h]
\begin{center}\psfrag{a}{$a$} \psfrag{b}{\hspace{.1cm}$b$} \psfrag{c}{$c$} \psfrag{d}{$d$}  \psfrag{b=d}{\hspace{-.13cm}$b=d$}\psfrag{A}{\hspace{-.7cm}$0<\lambda < \beta / 2$} \psfrag{B}{$\lambda= \beta / 2$} \psfrag{C}{\hspace{.2cm}$\beta / 2 < \lambda
< \beta$} \psfrag{D}{$ \lambda = \beta$} \psfrag{E}{$\lambda >
\beta$}\psfrag{1}{$12_3$}  \psfrag{2}{$11_3$}
\psfrag{3}{$10_3$}\psfrag{4}{$14_1$} \psfrag{5}{$15_1$}
\epsfxsize=12cm \epsfbox{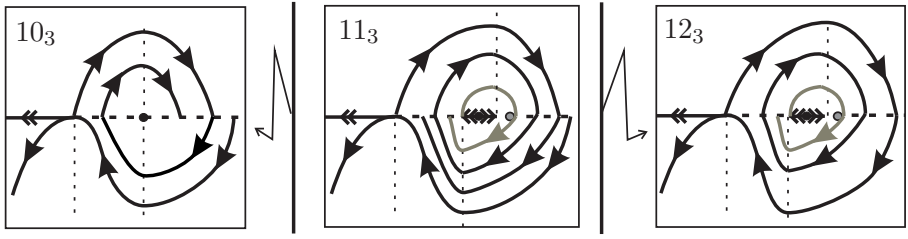}
\caption{\small{Cases $10_3 - 12_3$.}} \label{fig teo 3 casos 10 11
12}
\end{center}
\end{figure}

$\diamond$ \textit{Case $11_3$. $\sqrt{\beta}
 <\lambda< \sqrt{\beta}- \mu/2$:}  The points of $\Sigma$ on the left of
$a$ and on $(b,d)$ belong to $\Sigma^e$. The points on $(a,b)$ and
on the right of $d$ belong to $\Sigma^c$. The graph of $H$ is like
H-3 of Figure \ref{fig diag bif H}. We can prove that $p_1 \in
(b,d)$ is a $\Sigma-$repeller and $r_1$ is on the right of $d$.
Since $\varrho_{Y}(\varrho_{X}(a)) \in (a,b)$ there exists a point
$Q \in (\varrho_{Y}(\varrho_{X}(a)),\varrho_{Y}(d))$ such that
$\eta'(Q)=1$. So there exists a hyperbolic attractor canard cycle
$\Gamma$, of kind I, passing through $Q$. See Figure \ref{fig teo 3
casos 10 11 12}. By Lemma \ref{lema hipotese H}, in this Hopf
Bifurcation a unique canard
cycle arises. Moreover, 
there exists a hyperbolic repeller canard cycle
 $\Gamma$ of kind III passing through $a$ and $c$.

$\diamond$ \textit{Case $12_3$. $\lambda = \sqrt{\beta}- \mu/2$:}
The configuration on $\Sigma$ and the graph of $H$ are the same as
Case $11_3$.  The attractor canard cycle $\Gamma$ presented in the
previous case is persistent. Recall that this canard cycle is born
from the bifurcation of Case $10_3$. So, the radius of $\Gamma$ does
not tend to zero when $\lambda$ tends to $\sqrt{\beta} + \mu/2$.
Moreover, it appears a non hyperbolic canard cycle passing through
$a$ and $c$. See Figure \ref{fig teo 2 casos 10 11 12}.

$\diamond$ \textit{Case $13_3$. $\sqrt{\beta} - \mu/2 <\lambda<
L_1$, Case $14_3$. $\lambda = L_1$,  Case $15_3$. $L_1 < \lambda < 2
\sqrt{\beta} - \mu/2 $, Case $16_3$. $\lambda = 2 \sqrt{\beta} -
\mu/2$, Case $17_3$. $2 \sqrt{\beta} - \mu/2 < \lambda < 3
\sqrt{\beta} - \mu$, Case $18_3$. $\lambda = 3 \sqrt{\beta} - \mu$
and Case $19_3$. $\lambda
> 3 \sqrt{\beta} - \mu$:} The analysis of
these cases is done in a  similar way as Cases $11_1$, $12_1$,
$13_1$, $14_1$, $15_1$, $16_1$ and $17_1$, respectively.


%
The bifurcation diagram is illustrated in Figure \ref{fig diag bif
teo 2} replacing the number $2$ subscript by the number
$3$.\end{proof}


\section{Proof of Theorem A}\label{secao prova teorema A}

\begin{proof}[Proof of Theorem A] Since in Equation \eqref{eq fold-cusp 3 parametros
inicio}
 we can take $\mu \in (-\mu_0,\mu_0)$, from Theorems 1, 2 and 3 we derive that  it bifurcation diagram contains all the $55$ cases
described 
in Theorems 1, 2 and 3. But some of them are $\Sigma-$equivalent and
the  number of distinct topological behaviors is $23$. Moreover,
each topological behavior can be represented respectively by the
Cases $1_1$, $2_1$, $3_1$, $4_1$, $5_1$, $6_1$,  $7_1$, $8_1$,
$9_1$, $10_1$, $11_1$, $12_1$, $13_1$,  $14_1$, $15_1$, $16_1$,
$17_1$, $10_2$, $11_2$, $12_2$, $10_3$, $11_3$ and $12_3$.

The full behavior of the three$-$parameter family of NSDS's
expressed by Equation \eqref{eq fold-cusp 3 parametros inicio} is
illustrated in Figure \ref{fig diagrama bifurcacao teo A}  where we
consider a sphere around the point $(\lambda, \beta, \mu) = (0,0,0)$
with a small ray and so we make a stereographic projection defined
on the entire sphere, except the south pole. Still in relation to
this figure, the numbers pictured correspond to the occurrence of
the cases described in the previous theorems. As expected, the cases
$3_1$ and $3_2$ are not represented in this figure because they are,
respectively, the center and the south pole of the
sphere.\end{proof}

\begin{figure}[h!]
\begin{center}
\psfrag{A}{$1_2$}\psfrag{B}{$1_1$}\psfrag{C}{$1_3$}
\psfrag{D}{$6_3$}\psfrag{E}{$2_2$}
\psfrag{F}{$2_1$}\psfrag{G}{$6_2$}\psfrag{H}{$7_1$}
\psfrag{I}{$7_3$}\psfrag{J}{$5_3$}
\psfrag{K}{$3_3$}\psfrag{L}{$7_2$}\psfrag{M}{$2_3$}
\psfrag{N}{$8_1$} \psfrag{O}{$8_3$}
\psfrag{P}{$8_2$}\psfrag{Q}{$5_2$}\psfrag{R}{$6_1$}
\psfrag{S}{$5_1$} \psfrag{T}{$11_3$}
\psfrag{U}{$12_3$}\psfrag{V}{$10_3$}\psfrag{X}{$9_2$}
\psfrag{Y}{$9_1$} \psfrag{W}{$17_3$}
\psfrag{Z}{$9_2$}\psfrag{0}{$15_1$}\psfrag{1}{$11_2$}
\psfrag{2}{$10_2$} \psfrag{3}{$13_1$}
\psfrag{4}{$16_3$}\psfrag{5}{$13_3$}\psfrag{6}{$4_2$}
\psfrag{7}{$4_1$} \psfrag{8}{$12_2$} \psfrag{9}{$16_2$}
\psfrag{`}{$\hspace{-.0cm}\mu>0$}\psfrag{~}{$\mu =
0$}\psfrag{:}{$17_2$}\psfrag{@}{$18_2$}
\psfrag{#}{$9_3$}\psfrag{^}{$11_1$}\psfrag{&}{$14_3$}
\psfrag{*}{$12_1$}\psfrag{}{$$}
\psfrag{-}{$4_3$}\psfrag{_}{$16_1$}\psfrag{=}{$15_3$}
\psfrag{+}{$10_1$}\psfrag{}{$$} \psfrag{|}{$\beta
>0$}\psfrag{;}{$15_1$}\psfrag{,}{$\hspace{.4cm}\mu<0$}
\psfrag{<}{$14_1$} \psfrag{.}{$15_2$} \psfrag{>}{$17_2$}
\psfrag{?}{$\lambda < 0$}\psfrag{a}{$\beta =
0$}\psfrag{b}{$\lambda=0$}\psfrag{c}{$\lambda>0$} \psfrag{d}{$\beta
<0$}\psfrag{e}{$14_2$}\psfrag{f}{$13_2$}\psfrag{q}{$18_3$}
\psfrag{w}{$19_3$}\psfrag{p}{$16_1$} \psfrag{o}{$17_1$}
\psfrag{n}{$18_2$} \psfrag{g}{$14_2$}\psfrag{z}{$19_2$}
\epsfxsize=13cm \epsfbox{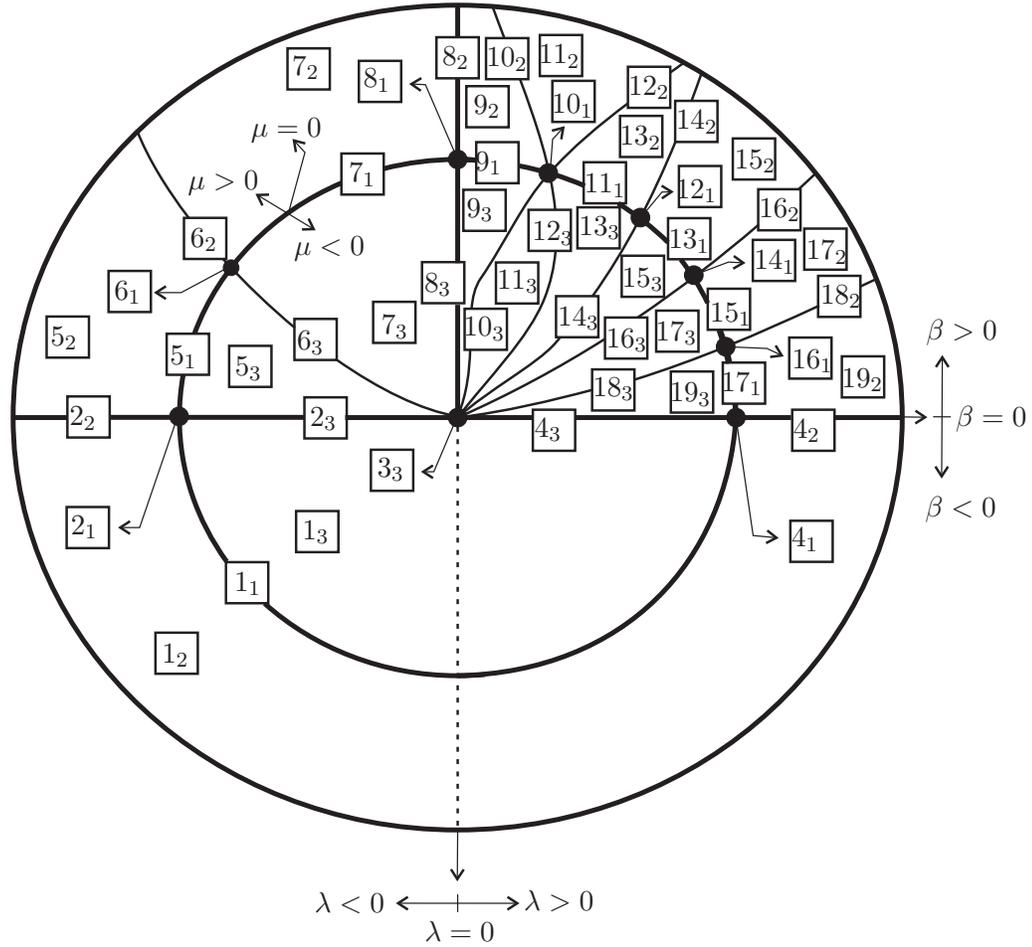}
\caption{\small{Bifurcation diagram of the invisible fold$-$cusp
singularity.}} \label{fig diagrama bifurcacao teo A}
\end{center}
\end{figure}


\section{Proof of Theorem B}\label{secao prova teorema B}

When we consider Equation \eqref{eq fold-cusp 2 parametros inicio}
the function $H$, given by \eqref{eq H}, is constant and equal to
$1$ independently of the value of $\mu$. Moreover, distinct values
of the bump function $\widetilde{B}$ (where $\widetilde{B} \neq B$)
do not produce any topological change in the bifurcation diagram of
the singularity. In another words, two parameters are enough to
describe the full behavior of this singularity. Observe that, by
Proposition \ref{prop quantidade de equilibria}, we have
$\Sigma^{f}=\emptyset$ and it does not have virtual pseudo
equilibria.


\begin{proof}[Proof of Theorem B] Since $X$ has a unique $\Sigma-$fold point
which is visible we conclude that canard cycles do not arise. In
Case $1_B$ we assume that $Y$ presents the behavior $Y^{-}$. In
Cases $2_B$, $3_B$ and $4_B$ we assume that $Y$ presents the
behavior $Y^0$. In Cases $5_B - 11_B$ we assume that $Y$ presents
the behavior $Y^+$.

$\diamond$ \textit{Case $1_B$. $\beta<0$:} The points of $\Sigma$ on
the left of $d$ belong to $\Sigma^c$ and the points on the right of
$d$ belong to $\Sigma^e$. See Figure \ref{fig 1 teo B}.

\begin{figure}[!h]
\begin{center}\psfrag{A}{$p_1$} \psfrag{B}{$p_2$} \psfrag{C}{$r_1$}\psfrag{D}{$r_2$}
\epsfxsize=3cm \epsfbox{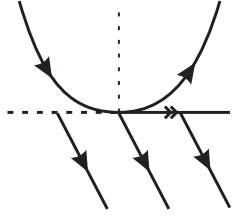} \caption{\small{Case
$1_B$.}} \label{fig 1 teo B}
\end{center}
\end{figure}

$\diamond$ \textit{Case $2_B$. $\lambda<0$, Case $3_B$. $\lambda=0$
and Case $4_B$. $\lambda>0$:} The configuration of the connected
components of $\Sigma$ is the same as Case $1_B$. Note that, when
$\lambda<0$ (Case $2_B$), it appears a tangential singularity
$P=(\lambda,0)\in \Sigma^e$ but $Z^{\Sigma}$ is always oriented from
the left to the right. These cases are illustrated in Figure
\ref{fig teo B casos 2 3 4}.

\begin{figure}[!h]
\begin{center}\psfrag{A}{$d$} \psfrag{B}{\hspace{-.1cm}$d=e$} \psfrag{C}{$r_1$}
\psfrag{D}{$r_2$}\psfrag{E}{$e$} \psfrag{F}{$d<e$} \psfrag{G}{$d=e$}
\psfrag{H}{$d>e$} \psfrag{I}{$2_1$} \psfrag{J}{$3_1$}
\psfrag{K}{$4_1$} \epsfxsize=11cm \epsfbox{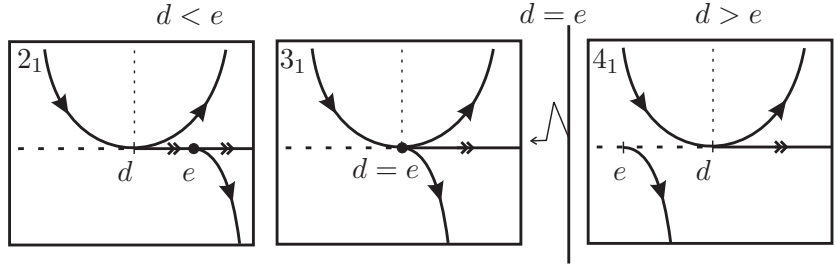}
\caption{\small{Cases $2_B - 4_B$.}} \label{fig teo B casos 2 3 4}
\end{center}
\end{figure}

$\diamond$ \textit{Case $5_B$. $\lambda< - 2 \sqrt{\beta}$, Case
$6_B$. $\lambda = - 2\sqrt{\beta}$ and Case $7_B$. $- 2\sqrt{\beta}
< \lambda < - \sqrt{\beta}$:} The points of $\Sigma$ on the right of
$b$ and inside the interval $(d,a)$ belong to $\Sigma^e$. The points
on $(a,b)$ and on the left of $d$ belong to $\Sigma^c$. See Figure
\ref{fig teo B casos 5 6 7}.

\begin{figure}[!h]
\begin{center}\psfrag{A}{$a$} \psfrag{B}{$b$} \psfrag{C}{$c$}
\psfrag{D}{$d$}\psfrag{E}{$0$} \psfrag{F}{$d<c$} \psfrag{G}{$d=c$}
\psfrag{H}{$c<d<a$} \psfrag{I}{$5_B$} \psfrag{J}{$6_B$}
\psfrag{K}{$7_B$} \psfrag{L}{$d=c$} \epsfxsize=11cm
\epsfbox{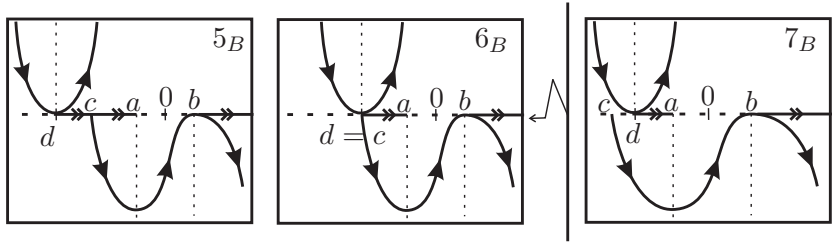} \caption{\small{Cases $5_B -
7_B$.}} \label{fig teo B casos 5 6 7}
\end{center}
\end{figure}

$\diamond$ \textit{Case $8_B$. $\lambda = -\sqrt{\beta}$:} In this
case $a=d$ and the configuration of the connected components of
$\Sigma$ is illustrated in Figure \ref{fig teo B caso 8}.

\begin{figure}[!h]
\begin{center}\psfrag{A}{$a=d$} \psfrag{B}{$b$} \psfrag{C}{$c$}
\psfrag{D}{$d$}\psfrag{E}{\hspace{.1cm}$0$} \psfrag{F}{$d<c$}
\psfrag{G}{$d=c$} \psfrag{H}{$c<d<a$} \psfrag{I}{$5_B$}
\psfrag{J}{$6_B$} \psfrag{K}{$7_B$} \psfrag{L}{$d=c$} \epsfxsize=3cm
\epsfbox{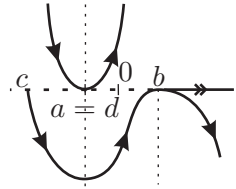} \caption{\small{Case $8_B$.}}
\label{fig teo B caso 8}
\end{center}
\end{figure}

$\diamond$ \textit{Case $9_B$. $-\sqrt{\beta}<\lambda<
\sqrt{\beta}$:} The points of $\Sigma$ on the right side of $b$
belong to $\Sigma^e$ and the points inside the interval $(a,d)$
belong to $\Sigma^s$. The points on $(d,b)$ and on the left of $a$
belong to $\Sigma^c$. See Figure \ref{fig teo B casos 9 10 11}.

$\diamond$ \textit{Case $10_B$. $\lambda = \sqrt{\beta}$:} In this
case $d=b$ and the configuration of the connected components of
$\Sigma$ is illustrated in Figure \ref{fig teo B casos 9 10 11}.

$\diamond$ \textit{ Case $11_B$. $\lambda > \sqrt{\beta}$:} The
points of $\Sigma$ on the right of $d$ belong to $\Sigma^e$ and the
points inside the interval $(a,b)$ belong to $\Sigma^s$. The points
on $(b,d)$ and on the left of $a$ belong to $\Sigma^c$. See Figure
\ref{fig teo B casos 9 10 11}.

\begin{figure}[!h]
\begin{center}\psfrag{A}{$a$} \psfrag{B}{$b$} \psfrag{C}{$c$}
\psfrag{D}{$d$}\psfrag{E}{$0$} \psfrag{F}{$d<c$} \psfrag{G}{$d=b$}
\psfrag{H}{$c<d<a$} \psfrag{I}{$9_B$} \psfrag{J}{$10_B$}
\psfrag{K}{$11_B$} \psfrag{L}{$d=b$} \epsfxsize=11cm
\epsfbox{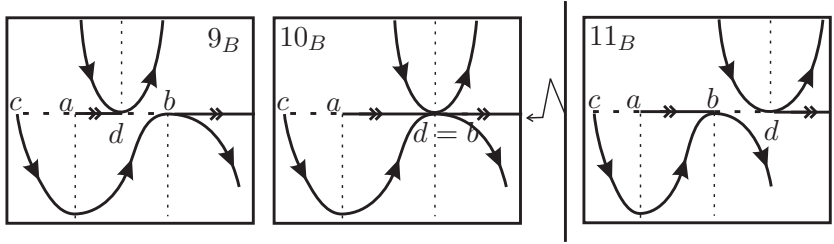} \caption{\small{Cases $9_B -
11_B$.}} \label{fig teo B casos 9 10 11}
\end{center}
\end{figure}

\begin{figure}[!h]
\begin{center}\psfrag{A}{$11_1$} \psfrag{B}{$11_B$} \psfrag{C}{$6_B$} \psfrag{D}{$14_1$} \psfrag{E}{$15_1$}
\psfrag{F}{$\lambda$}  \psfrag{G}{$\beta$}
\psfrag{H}{$\lambda=-\sqrt{\beta}$} \psfrag{I}{$\lambda=
2\sqrt{\beta}$} \psfrag{J}{$\lambda= -2\sqrt{\beta}$}
\psfrag{K}{$\lambda=\sqrt{\beta}$}
 \psfrag{1}{$1_B$} \psfrag{2}{$2_B$} \psfrag{3}{$3_B$}
  \psfrag{4}{$4_B$} \psfrag{5}{$7_B$} \psfrag{6}{$8_B$}
   \psfrag{7}{$5_B$} \psfrag{8}{$9_B$} \psfrag{9}{$9_1$}
    \psfrag{0}{$10_B$}
\epsfxsize=8cm \epsfbox{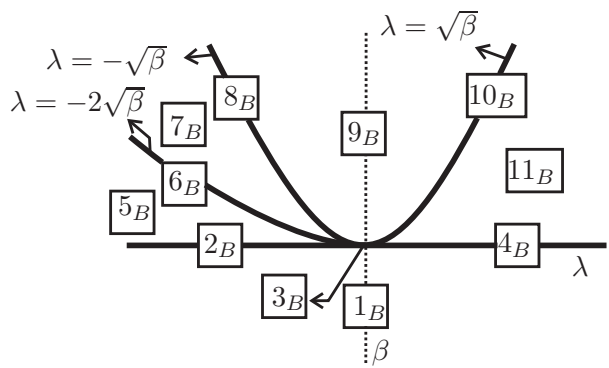}
\caption{\small{Bifurcation Diagram of Theorem B.}} \label{fig diag
bif teo B}
\end{center}
\end{figure}

The bifurcation diagram is illustrated in Figure \ref{fig diag bif
teo B}.\end{proof}

\section{Concluding Remarks}\label{secao conclusao}

The results in Section 12 of \cite{Marcel} were revisited and
extended in this paper. The bifurcation diagram of a
three$-$parameter family of NSDS's presenting a fold$-$cusp
singularity is exhibited. In particular it is shown the existence of
some new interesting global bifurcations around the standard
fold$-$cusp singularity expressed by \eqref{eq fold-cusp geral}.
Moreover, the simultaneous occurrence of such local and global
bifurcations indicates how complex is the behavior of this
singularity.


\vspace{1cm}

\noindent {\textbf{Acknowledgments.} We would like to thank the
referee  for helpful comments and suggestions. The first and the
third authors are partially supported by a FAPESP-BRAZIL grant
2007/06896-5. The second author was partially supported by
FAPESP-BRAZIL grants 2007/08707-5, \linebreak 2010/18190-2 and
2012/00481-6.}

\end{document}